\documentclass[12pt]{article}

\usepackage{GoetzArticle}
\usepackage{amssymb}
\usepackage{verbatim}
\usepackage{amssymb}
\usepackage{amsmath}
\usepackage{rotating}
\usepackage{bm}

\newcommand{\Z}{\mathbb{Z}}

\newcommand{\supp}{{\rm supp}\, }
\newcommand{\sgn}{{\rm sgn}}

\newcommand{\N}{\mathbb{N}}

\newcommand{\C}{\mathbb{C}}

\renewcommand{\P}{{\mathbb{P}}}
\newcommand{\E}{{\mathbb{E}}}

\newcommand{\Tr}{{\operatorname{Tr}}}

\def\CC{{\mathbb C}}

\newcommand{\newtimes}{\cdot}
\newcommand{\sparsity}{S}
\newcommand{\transk}{k}

\title{Sparsity in Time--Frequency Representations}
\author{
        G\"otz E. Pfander\footnotemark[1]~ and
        Holger Rauhut\footnotemark[2]
        }

\begin{document}
\maketitle
\abstract{We consider signals and operators in finite dimension which have sparse time-frequency representations. As main result we show that an $\sparsity$-sparse Gabor representation in $\CC^n$ with respect to a random unimodular window can be recovered by Basis Pursuit with high probability provided that $\sparsity \leq Cn/\log(n)$. Our results are applicable to the channel estimation problem in wireless communications and they establish the usefulness of a class of measurement matrices for compressive sensing.}

\renewcommand{\thefootnote}{\fnsymbol{footnote}}

\footnotetext[1]{School of Engineering and Science, Jacobs
University Bremen, 28759 Bremen, Germany, g.pfander@jacobs-university.de}

\footnotetext[2]{Numerical Harmonic Analysis Group, Faculty of
Mathematics, University of Vienna, Nordbergstrasse 15, A-1090
Vienna, Austria, holger.rauhut@univie.ac.at\\
H.R. acknowledges support by the European Union's Human Potential
Programme through an Individual Marie Curie Fellowship, contract number
MEIF CT-2006-022811}

\renewcommand{\thefootnote}{\arabic{footnote}}
\pagenumbering{arabic} \maketitle

\numberwithin{equation}{section}



\section{Introduction}

Efficient algorithms aiming at the recovery of signals and operators
from a restricted number of measurements must be based on
some {\it a-priori}\, information about the object under investigation.
In a large body of recent work, the signal or operator at hand is assumed to
have a sparse representation in a given dictionary. A typical example in
this realm is the recovery of vectors that are
sparse in the Euclidean basis, that is, of vectors which have a limited number
of nonzero components at unknown locations. Such a vector is to be
determined efficiently by a small number of linear measurements which are given
by inner products with appropriately chosen analysis vectors.

The difficulty in this body of work lies in the fact that sparsity
conditions as those mentioned above define nonlinear subspaces of
linear signal or operator spaces. To circumvent a combinatorial and
therefore unfeasible exhaustive search, efficient alternatives such
as $\ell_1$-minimization (Basis Pursuit) 
and greedy algorithms such as Matching Pursuits
have been proposed 
in the sparse representations and compressed sensing literature,
see, for example, \cite{chdosa99,do04,carota06,carota06-1,badadewa06,kura06,ra05-7,gitr05,tr04}.
In compressed sensing one commonly uses linear random measurements
for the recovery of a sparse signal
with high probability.
So far, mainly random Gaussian, Bernoulli and partial Fourier measurements
have been considered successfully
\cite{cata06,do04,badadewa06,rascva06,ruve06}.
A typical result states that a signal of length $N$ with at most $\sparsity$ non-zero
entries can be recovered
from $n$ randomly selected samples of its Fourier transform with
high probability provided $\sparsity \leq C n/\log(N)$ \cite{carota06,ra05-7}.

In this paper, we consider sparse representations
in terms of time--frequency shift dictionaries, and investigate
recovery conditions similar to the ones for Gaussian, Bernoulli and Fourier measurements.
Here, $T_k$ denotes the cyclic shift respectively translation
operator and $M_\ell$  the modulation operator respectively frequency shift
operator on $\CC^n$, defined by \begin{equation}\label{eq:trans_mod}
(T_\transk h)_q=h_{\transk+q \mod n}\quad\mbox{and}\quad (M_{\ell}h)_q = e^{2\pi i\ell
q/n} h_q.
\end{equation}
Then $\pi(\lambda) = M_\ell T_\transk$, $\lambda = (\transk,\ell)$,
is a time-frequency shift and the system
$\{\pi(\lambda): \lambda \in \Z_n{\times}\Z_n\}$, $\Z_n = \{0,1,\hdots,n-1\}$,
of all time-frequency shifts
forms a basis of $\CC^{n{\times} n}$ \cite{LPW05,krpfra07}. For a non-zero vector $g$, the so-called window, the set
\begin{equation}\label{GaborSystem}
\{\pi(\lambda) g: \lambda \in \Z_n{\times} \Z_n\}
\end{equation}
is called a Gabor system \cite{Gro01} and the matrix
$\Psi_g \in \C^{n {\times} n^2}$ whose columns are the members
$\pi(\lambda) g$, $\lambda \in \Z_n{\times} \Z_n$ of a Gabor system
is referred to as Gabor synthesis matrix. The Gabor system given in
\eqref{GaborSystem} is a tight frame in $\C^n$ whenever
$g\neq 0$ \cite{LPW05,Chr03}.

A vector $x$ is called $\sparsity$-sparse if it has at most $\sparsity$ non-zero
entries; formally
$\|x\|_0 := |\supp x| = 
\#\{\lambda: x_\lambda \neq 0\} \leq \sparsity$.

Our analysis of sparsity in conjunction with time--frequency shift dictionaries addresses the following, clearly equivalent objectives.
\newcounter{Objcount}
\begin{list}{\bf Objective \arabic{Objcount}. } {\usecounter{Objcount} \setlength{\leftmargin}{0cm} \setlength{\itemindent}{1cm}  \setlength{\listparindent}{.6cm} \setlength{\itemsep}{.3cm}}%

 \item  Determine the coefficient sequence of a vector
            \begin{equation}\label{TFrep}
                    y = \sum_{\lambda \in \Z_n{\times} \Z_n} x_{\lambda} \pi(\lambda) g
            \end{equation}
            that is known to have a sparse representation in the Gabor system $\{\pi(\lambda) g: \lambda \in \Z_n{\times} \Z_n\}$ with window $g\neq 0$. Clearly, the representation (\ref{TFrep}) is redundant; given $y$ the
            coefficient vector $x$ is not unique and it is a non-trivial problem of
            computing efficiently the sparsest representation of $y$.

            If $g$ is well localized in time and frequency, then the sparse coefficient vector $x$ can be seen to describe the time--frequency content of any signal $y=\Psi_g x$ \cite{Gro01}. Note that the windows
(\ref{eq:gA}), (\ref{eq:gR})
considered in this paper are neither well localized in time nor in frequency.

        \item  Establish the applicability of $\Psi_g$ as measurement matrix for compressed sensing, that is, consider the rows of $\Psi_g$ as measurement vectors, in the classical strategy of efficiently determining a signal $x$ which is sparse in the Euclidean basis. In short, the aim is again to recover $x$ from $y=\Psi_g x$ whenever $\|x\|_0$ can be assumed small.
            
            The window vector $g$ used to achieve our main results Theorems~\ref{thm:randomphase} and \ref{thm:recover} is chosen at random \eqref{eq:gR}, that is, $\Psi_g\in\C^{n{\times}n^2}$ depends on $n$ independent random variables as 
compared to $n{\times}N$ independent 
random variables in the case of Gaussian or Bernoulli 
$n \times N$ measurement matrices \cite{cata06,badadewa06}. Note that 
our results apply also to $n\times N$ measurement matrices, $N \leq n^2$, 
that are obtained by removing $n^2-N$ columns from $\Psi_g$.

Further, the 
structure of $\Psi_g$ allows for fast Fourier transform based matrix vector
multiplication algorithms \cite{Str98} (in contrast to unstructured
Gaussian or Bernoulli random matrices). This leads to efficient implementations
of $\ell_1$-minimization methods \cite{bova04}.
            
        \item  Identify from a single input output pair $(g,\Gamma g)$ the coefficient vector $x$ of an operator
        \begin{equation}\label{def:Gamma}
                    \Gamma = \sum_{\lambda \in \Z_n{\times} \Z_n} x_\lambda \pi(\lambda),
        \end{equation}
        where $\Gamma$ is assumed to have a sparse representation in the system of time frequency shift matrices $\{\pi(\lambda) : \lambda \in \Z_n{\times} \Z_n\}$.

In short, the task at hand is to identify
$\Gamma \in \C^{n{\times} n}$, or equivalently $x$, from its action $y= \Gamma g$ on a single vector $g$.
Writing
\[
y = \Gamma g \,=\, \sum_{\lambda \in \Z_n {\times} \Z_n} x_\lambda \pi(\lambda) g
\]
with unknown but sparse $x$, we observe the equivalence of this objective  with Objectives~1 and 2.

This objective falls in the realm of what is known as channel operator estimation / identification in communications engineering, and, indeed, \eqref{def:Gamma}
is a common model of wireless channels \cite{be63,GP07,Cor01,Pae01} and sonar \cite{st99,mi87} where physical considerations often suggest that $x$ is rather sparse. First results were obtained in \cite{pfrata07}, on which we will improve here. Further, using multiple input output pairs for the efficient recovery of channel operators is discussed in \cite{pfrata07}.
Sparsity constraints in the dictionary of time--frequency shifts have also been considered for
radar applications \cite{HS07}.

\end{list}

In the following, we will phrase our results in terms of Objective~1, namely,
we assume that $y$ is given and has an unknown $\sparsity$-sparse
representation (\ref{TFrep}) in a given Gabor system \eqref{GaborSystem}
with $\sparsity < n$.

A natural strategy to recover the corresponding coefficient vector in this setup consists in
seeking the vector $x$ with minimal
support consistent with $y$; in other words solving
the $\ell_0$-minimization problem
\begin{equation}\label{def:l0}
\min_{x} \|x\|_0 \quad \mbox{subject to } \Psi_g x = y.
\end{equation}
Unfortunately, this problem is NP hard in general \cite{avdama97}, and hence,
is not feasible in practice. In order to avoid this computational
bottleneck, several alternative reconstruction methods have been
suggested as mentioned above.
We will concentrate
here on Basis Pursuit, which seeks the solution of the convex
problem
\begin{equation}\label{def:Basis Pursuit}
\min_{x} \|x\|_1 \quad \mbox{subject to } \Psi_g x = y,
\end{equation}
where $\|x\|_1 = \sum_{\lambda \in \Z_n^2} |x_{\lambda}|$ is the
$\ell_1$-norm of $x$.
This problem can be solved with efficient convex optimization techniques
\cite{bova04,chdosa99,dots06-1}. Of course, the hope is that the solution of (\ref{def:Basis Pursuit})
coincides with the solution of the
$\ell_0$-minimization problem (\ref{def:l0}).
It is the goal of this paper to make this rigorous.

So far we did not specify the window $g$ in \eqref{GaborSystem}. In \cite{pfrata07} we proposed
to work with the Alltop window $g^A$ \cite{al80,hest03} with entries
\begin{equation}\label{eq:gA}
g^A_q \,=\, \frac{1}{\sqrt{n}}e^{2\pi i q^3/n}, \quad q=0,\hdots,n{-}1,
\end{equation}
and with the randomly generated window $g^R$ with entries
\begin{equation}\label{eq:gR}
g^R_q \,=\, \frac{1}{\sqrt{n}} \epsilon_q, \quad q=0,\hdots,n{-}1,
\end{equation}
where the $\epsilon_q$ are independent and uniformly distributed
on the torus $\{z \in \CC, |z|=1\}$; in other words, $g^R$ is a
normalized Steinhaus sequence.
The Alltop window will only be used for prime $n \geq 5$. Although both 
windows seem to be a bit unfamiliar
in terms of time-frequency analysis due to their lack of time-frequency
concentration (they are actually completely unlocalized in both
time and frequency), they may perfectly be applied to the
problem of identifying a sparse operator $\Gamma$ of the
form (\ref{def:Gamma}) in Objective~3 \cite{pfrata07}.

In \cite{pfrata07}, the following theorem concerning the recovery of sparse
time-frequency representations in terms of $g^A$ and $g^R$ was shown.
\begin{theorem}\label{cor1}
\begin{itemize}
\item[(a)] Let $n\geq 5$ be prime and $g=g^A$ be the Alltop window
defined in (\ref{eq:gA}). If $\sparsity < \frac{\sqrt{n}+1}{2}$ then Basis Pursuit given in \eqref{def:l0}
recovers all $\sparsity$-sparse $x$ from $y=\Psi_{g} x$.
\item[(b)] Let $n$ be even and choose $g=g^R$ to be the random unimodular
window in (\ref{eq:gR}). Let $t > 0$ and suppose
\begin{equation}\label{cond:k}
\sparsity \leq \frac{1}{4} \sqrt{\frac{n}{2 \log n + \log 4 + t}} + \frac{1}{2}\,.
\end{equation}
 Then with probability at least
$1-e^{-t}$ Basis Pursuit \eqref{def:l0} recovers all $\sparsity$-sparse $x$ from $y= \Psi_{g}$.
\end{itemize}
\end{theorem}

Theorem~\ref{cor1} is based on standard recovery results for Basis Pursuit
which rely on the coherence of $\Psi_g$ \cite{tr04,doelte06}. The coherence for $g=g^A$ was
given in \cite{hest03}, and the one for $g=g^R$
was estimated in \cite{pfrata07}, see (\ref{coherence:gR}).
Although Theorem~\ref{cor1} shows that recovery guarantees
can be given, the conditions on
the maximal sparsity $\sparsity$ are quite restrictive; $\sparsity$
has to be as small as
of the order of $\sqrt{n}$ or even $\sqrt{n/\log(n)}$.

Passing from a worst case
analysis to an average case analysis in the sense that the support set
of $x$ and the signs of its non-zero coefficients are chosen at random,
it is possible to apply recent results of Tropp \cite{tr06-2} to show that
recovery can be ensured with high probability provided
\begin{equation}\label{cond:k2}
\sparsity \leq C\, \frac{n} {\log(n)^{u}}
\end{equation}
for some constant $c$ where $u =1$ in the case of $g^A$ and $u =2$ in
the case of  $g^R$. For a precise formulation of these results,
see Theorem~2.5 in  \cite{pfrata07}.

In this paper we will work with the
randomly generated window $g^R$ and gradually improve conditions
(\ref{cond:k}) and (\ref{cond:k2}) to $\sparsity \leq C n/\log(n)$, while
removing the randomness assumption on the coefficients $x$.
It seems rather difficult to perform a similar task for the deterministic
Alltop window $g^A$.

The paper is organized as follows. In Section \ref{Sec:Results} we state
our two main results on recovery of sparse time-frequency representations, namely Theorems~\ref{thm:randomphase} and \ref{thm:recover}.
Section \ref{Sec:Cond} will deal with the estimation of the smallest
and largest
singular value of a submatrix of $\Psi_g$ for $g=g^R$, which plays a central role in the proofs of Theorems~\ref{thm:randomphase} and \ref{thm:recover}. In Section \ref{Sec:RPh}
we prove Theorem~\ref{thm:randomphase} on recovery of sparse
coefficients $x$ with random phases; while Section \ref{Sec:Det} contains the proof
of Theorem~\ref{thm:recover} on the recovery of deterministic sparse coefficients $x$.

Throughout the paper $\|\cdot\|_p$ denotes the usual $\ell_p$-norm
on sequences, while $\|\cdot\|_{p\to q}$ is the operator norm from $\ell_p$
to $\ell_q$, and, for brevity $\|\cdot\| = \|\cdot\|_{2\to 2}$. The Frobenius
norm of a matrix $A$ is defined as $\|A\|_F = \sqrt{\Tr(A^*A)}$, where $\Tr$ is
the trace.
Furthermore,
$\P\big(E\big)$ denotes the probability of an event $E$ and $\E$ means expectation.

\section{Statement of Results}
\label{Sec:Results}

Our results are concerned with the recovery by Basis Pursuit (\ref{def:Basis Pursuit})
of sparse time-frequency
representations (\ref{TFrep}) with the randomly generated
window $g=g^R$ given in (\ref{eq:gR}). We present a first result for
deterministic support sets, that is, for every possible support set, but random phases of the coefficient
vector $x$; and a second result, Theorem~\ref{thm:recover}, for deterministic $x$.

\begin{theorem}\label{thm:randomphase}
Let $n$ be even, and let $\Lambda \subset \Z_n \times \Z_n$ be of cardinality
$|\Lambda| = \sparsity$.
Let $x$ with $\supp(x) = \Lambda$ be such that on $\Lambda$ the random phases
$(\sgn(x_\lambda))_{\lambda \in \Lambda}$ are independent and uniformly distributed
on the torus $\{z \in \C, |z| = 1\}$. Let $\sigma > 8$. Choose the window
$g=g^R$ as in (\ref{eq:gR}), that is, with random
entries independently and uniformly distributed on the torus
$\{z \in \C, |z| = 1\}$.
Then with probability at most
\[
2(n^2-\sparsity)\exp\left(- \frac{n}{8 \sigma \sparsity \log n}\right)
+ CS \exp\left(- \frac{n}{16e \sparsity}\right)
+ 4 n^{-(\sigma/4 - 2)}
\]
Basis Pursuit (\ref{def:Basis Pursuit}) fails to recovers $x$ from $y = \Psi_g x$. Here, the constant $C \approx 1.075$.
\end{theorem}

\begin{remark} Note that the probability estimate above becomes
effective once
\[
n  \gtrsim \max\{ 16 e \sparsity \log(C \sparsity), 64\sparsity \log(n) \log(2n^2)\},
\]
or even simpler, if $\sparsity \leq C_0 \frac{n}{\log^2(n)}$ 
for appropriately chosen $C_0$.
\end{remark}

The restriction to $n$ even was made for the sake of simple exposition;
a similar result holds also for $n$ odd
(compare also to Theorem~5.1 in \cite{pfrata07}).

Recovery is also possible for
deterministic sparse coefficients.
The corresponding proof is more involved, however.

\begin{theorem}\label{thm:recover}
Assume $x$ is an arbitrary $\sparsity$-sparse coefficient vector. Choose
the random unimodular Gabor window $g=g^R$ defined in (\ref{eq:gR}), that is, with random
entries independently and uniformly distributed on the torus
$\{z \in \C, |z| = 1\}$.
Assume that
\begin{equation}\label{cond:nk}
\sparsity \leq C \frac{n}{\log(n/\varepsilon)}
\end{equation}
for some constant $C$ (see Remark~\ref{remark:constant}). Then with probability at least $1-\varepsilon$
Basis Pursuit (\ref{def:Basis Pursuit}) recovers $x$ from $y = \Psi x = \Psi_g x$.
\end{theorem}

\begin{remark}\label{remark:constant}
From the proof of Theorem~\ref{thm:recover} one can deduce  information about
the constant in (\ref{cond:nk}). Indeed recovery is ensured provided
\[
n \geq \max \{C_1 \sparsity \log(n^2/\varepsilon), C_2 \sparsity (\log(\sparsity^4/\varepsilon) + C_3) \}
\]
with $C_1 = 273.5$, $C_2 = 64.1$ and $C_3 = 8.35$. Hence, the constant
of Theorem~\ref{thm:randomphase} is better than those in Theorem~\ref{thm:recover},
but this improvement comes at the cost of a worse exponent at the logarithm 
and of assuming random phases $\sgn(x_\lambda)$.
\end{remark}

Numerical experiments illustrating our recovery results were already
given in \cite{pfrata07}; clearly, they can only indicate the average
case behaviour rather than the worst case behaviour covered in
Theorem~\ref{thm:recover}. These experiments suggest that most $\sparsity$-sparse
signals can be recovered provided $\sparsity \leq \frac{n}{2\log(n)}$. So
Theorem \ref{thm:recover} seems to indicate
the right asymptotic order $n/ \log(n)$, but the constants are likely not
optimal.

We note once more that both theorems can be interpreted as compressed sensing
type results on recovery of $\sparsity$-sparse vectors in $\C^{n^2}$ from
$n$ measurements with $\Psi_g \in \CC^{n{\times} n^2}$ playing the role
of the (random) measurement matrix as described in Objective~2. Also, both results can be applied to identify matrices which have a sparse representation in the basis of time--frequency shift matrices as described in Objective~3.

Furthermore, Theorems~\ref{thm:randomphase} and \ref{thm:recover}
hold literally (including their proofs) when we pass from $\Z_n$ to time-frequency analysis on an arbitrary finite Abelian group;
in particular, on multi-dimensional versions $\Z_n^d$ with $d\geq 1$ where $n$ would be replaced by $n^d$ in all of the statements.

\section{Well conditioned  submatrices of Gabor synthesis matrices}
\label{Sec:Cond}
It is crucial for sparse recovery that small column submatrices of measurement or synthesis matrices such as $\Psi_g$
are well-conditioned. Before proceeding to the proofs of our main Theorems
\ref{thm:randomphase} and \ref{thm:recover}
we will deal with such an analysis in this section.

Throughout the rest of the paper we let $\Psi=\Psi_g \in \CC^{n{\times} n^2}$
with $g=g^R$ being the randomly generated unimodular window described in (\ref{eq:gR}).
For $\Lambda\subseteq \Z_n{\times}\Z_n$ and $A\in \CC^{n{\times} n^2}$ we denote by
$A_\Lambda$ the matrix consisting only of those columns indexed
by $\lambda\in\Lambda$.

\begin{theorem}\label{thm:opnorm} Let $\varepsilon, \delta \in (0,1)$ and
$|\Lambda|=\sparsity$. Suppose that
\begin{equation}\label{cond1}
\sparsity \leq \frac{\delta^2 n}{4e(\log(\sparsity/\varepsilon) + c)}
\end{equation}
with $c = \log(e^2/(4(e{-}1))) \approx 0.0724$. Then $\|I_\Lambda -  \Psi_\Lambda^* \Psi_\Lambda\| \leq \delta$
with probability at least $1-\varepsilon$; in other words the minimal and maximal
eigenvalues
of $\Psi_\Lambda^* \Psi_\Lambda$ satisfy
$
1-\delta \leq \lambda_{\min} \leq \lambda_{\max} \leq 1+\delta
$
with probability at least $1-\varepsilon$.
\end{theorem}

\begin{remark} Assuming equality in condition (\ref{cond1}) and solving for $\varepsilon$ we deduce
\begin{equation}\label{ineq:op2}
\P\big( \|I_\Lambda - \Psi_\Lambda^* \Psi_\Lambda\| > \delta)
\leq \frac{e^2}{4(e{-}1)} \sparsity \exp\left(-\frac{\delta^2 n}{4e \sparsity}\right)
= C \sparsity \exp\left(-\frac{\delta^2 n}{4e \sparsity}\right)
\end{equation}
with $C \approx 1.075$.
\end{remark}

In the following we will develop the proof of Theorem~\ref{thm:opnorm}.

\subsection{Expectation of a Frobenius norms}

We set $H=\Psi_\Lambda^\ast\Psi_\Lambda - I_\Lambda$.
An important step towards Theorem~\ref{thm:opnorm} is to estimate
the expectation of the Frobenius norm 
of powers of $H$. Indeed having accomplished this task one may
use Markov's inequality, the fact that the Frobenius
norm majorizes the operator norm, and the fact that $H$ is self-adjoint to obtain
\begin{align}
        \P\big(\|I_\Lambda-\Psi_\Lambda^\ast\Psi_\Lambda\|> \delta\big)
            \, &=\, \P\big(\|H\|> \delta\big)
                =\P\big(\|H\|^{2m} > \delta^{2m}\big)
\leq  \delta^{-2m} \E[ \|H\|^{2m} ]\notag\\[.3cm]
               &  = \delta^{-2m} \E[ \|H^{m}\|^2 ]
                \leq \delta^{-2m} \E[ \|H^m\|_F^{2} ]
            = \delta^{-2m} \E[ {\Tr} H^{2m} ]. \label{Pineq1}
\end{align}
We will use the following concept to estimate $\E[\Tr (H^{2m})]$.
\begin{definition}
 The associated Stirling number of the first kind, denoted by
$d_2(m,s)$, is the number of permutations of $m$ elements which involve exactly $s$
disjoint cycles and where each cycle has at least 2 elements.
\end{definition}

The associated Stirling numbers satisfy the following recursion
\cite[p.\ 75]{ri58}
\begin{equation}\label{Stirling:rec}
d_2(m{+}1,s)=m[d_2(m,s)+d_2(m{-}1,s{-}1)], \quad 
1 \leq s \leq m/2\, ,
\end{equation}
with boundary conditions
\begin{equation}\label{bound:cond}
d_2(0,0) = 1, \quad d_2(m,0)=0, \quad d_2(m,s)=0,\quad m \geq 1, s > m/2.
\end{equation}
Equipped with this tool, the desired expectation of the Frobenius
norm in \eqref{Pineq1}  can be estimated as follows.

\begin{lemma}\label{lem:frob} If $\sparsity = |\Lambda|$ and $m$ even then
\begin{equation}\label{equation:LemmaCount1}
\E[\Tr H^{m}] \leq \sparsity \left(\frac{\sparsity} n \right)^{m} \ \ \sum_{s=1}^{m/2} d_2(m,s) \,\left(\frac n {\sparsity} \right)^s.
\end{equation}
\end{lemma}
\begin{proof} Note that for $\lambda_j\in\Lambda$, we have
\begin{eqnarray*}
    H_{\lambda_1,\lambda_2}&=& \left\{
                               \begin{array}{ll}
                                  \langle \pi(\lambda_1)g,\pi(\lambda_2)g \rangle, & \hbox{if } \lambda_1\neq\lambda_2 ,\\
                                 0, & \hbox{if } \lambda_1=\lambda_2,
                               \end{array}
                             \right. \\
 H^2_{\lambda_1,\lambda_3}&=&\sum_{\lambda_2} H_{\lambda_1,\lambda_2}H_{\lambda_2,\lambda_3}
                =\sum_{\lambda_2\neq \lambda_1,\lambda_3}  \langle \pi(\lambda_1)g,\pi(\lambda_2)g \rangle \,\langle \pi(\lambda_2)g,\pi(\lambda_3)g \rangle\,,
   \\
   H^3_{\lambda_1,\lambda_4}&=& \sum_{\lambda_3} H_{\lambda_1,\lambda_3}^2 H_{\lambda_3,\lambda_4}
                  =\sum_{\lambda_3\neq \lambda_4} \sum_{\lambda_2\neq \lambda_1,\lambda_3}  \langle \pi(\lambda_1)g,\pi(\lambda_2)g \rangle\, \langle \pi(\lambda_2)g,\pi(\lambda_3)g \rangle
  \, \langle \pi(\lambda_3)g,\pi(\lambda_4)g \rangle\,,
\end{eqnarray*}
and, in general,
$$
 H^{m}_{\lambda_1,\lambda_{m{+}1}}
                =\sum_{\lambda_2\neq \lambda_1,\lambda_3} \sum_{\lambda_3\neq \lambda_4}
  \cdots \sum_{\lambda_m\neq \lambda_{m{+}1}} \langle \pi(\lambda_1)g,\pi(\lambda_2)g \rangle\, \langle \pi(\lambda_2)g,\pi(\lambda_3)g \rangle
\cdots \langle \pi(\lambda_m)g,\pi(\lambda_{m{+}1})g \rangle\,.
$$
Consequently,
\begin{eqnarray*}
  \Tr H^m
                &=&\sum_{\lambda_1} \sum_{\lambda_2\neq \lambda_1,\lambda_3} \sum_{\lambda_3\neq \lambda_4}
  \cdots \sum_{\lambda_m\neq \lambda_{1}} \langle \pi(\lambda_1)g,\pi(\lambda_2)g \rangle\, \langle \pi(\lambda_2)g,\pi(\lambda_3)g \rangle
\cdots \langle \pi(\lambda_m)g,\pi(\lambda_{1})g \rangle \\
&=&
\sum_{\substack{\lambda_1,\ldots,\lambda_m\in\Lambda
\\ \lambda_1\neq\lambda_2\neq \lambda_3\neq\cdots\neq \lambda_m\neq\lambda_1}}
\langle \pi(\lambda_1)g,\pi(\lambda_2)g \rangle\, \langle \pi(\lambda_2)g,\pi(\lambda_3)g \rangle
\cdots \langle \pi(\lambda_m)g,\pi(\lambda_{1})g \rangle\,.
\end{eqnarray*}
Linearity of $\E$  implies that $\E[ {\rm Tr} H^{m} ]=\sum_{\lambda_1\neq\lambda_2\neq \lambda_3\neq\cdots\neq \lambda_m\neq\lambda_1}  E_{\lambda_1,\ldots,\lambda_m}$ where
\begin{eqnarray}
E_{\lambda_1,\ldots,\lambda_m}=
\E\left[\langle \pi(\lambda_1)g,\pi(\lambda_2)g \rangle\, \langle \pi(\lambda_2)g,\pi(\lambda_3)g \rangle
\cdots \langle \pi(\lambda_m)g,\pi(\lambda_{1})g \rangle
    \right] \,.\label{equation:expectedvalueinnerproducts}
\end{eqnarray}
We denote $\lambda_\alpha = (k_\alpha,\ell_\alpha)$ with
$k_\alpha, \ell_\alpha \in \Z_n$, $\alpha = 1,\hdots,n$.
Applying once more linearity of $\E$ to the inner products in
(\ref{equation:expectedvalueinnerproducts})
we obtain
\begin{eqnarray}
&& \hspace{-1.2cm}E_{\lambda_1,\ldots,\lambda_m}= \notag \\
    &&\hspace{-1cm}\sum_{j_1} \sum_{j_2} \ldots \sum_{j_{m}}
e^{2\pi i j_1(\ell_1{-}\ell_2)/n} e^{2\pi i j_2(\ell_2{-}\ell_3)/n}
    \ldots e^{2\pi i j_{m}(\ell_{m}{-}\ell_{1})/n}\,  \notag
 \notag \\
&& \hspace{-1cm}\newtimes \E\left[
        g(j_1{-}k_1) \overline{g(j_1{-}k_2)} \,
        g(j_2{-}k_2) \overline{g(j_2{-}k_3)}
        \ldots
         g(j_{m{-}1}{-}k_{m{-}1})  \overline{g(j_{m{-}1}{-}k_{m})}\,
         g(j_{m}{-}k_m)  \overline{g(j_{m}{-}k_{1})}
    \right]. \label{equation:sumsexpectedvaluesummands}
\end{eqnarray}
Here and throughout the remainder of the paper, addition and subtraction of
indices $j_1-k_1$ etc.\ is understood modulo $n$.

The independence of the $g(j)$ implies that the summands in
\eqref{equation:sumsexpectedvaluesummands} factor into a product of
expectations
over powers of $g(j)$'s, namely, into factors of the form
$\E\left[  g(j)^{u_j} \overline{g(j)^{v_j}} \right]$, $u_j,v_j\in \N$.
As $\E[g(j)]=0$ and, by unimodularity of $\sqrt{n}\, g$, $\E[g(j)\overline{g(j)}]=\tfrac 1 n$, we have $\E\left[  g(j)^{u_j} \overline{g(j)^{v_j}} \right]=0$ if $u_j\neq v_j$  and $\E\left[  g(j)^{u_j} \overline{g(j)^{u_j}} \right]=n^{-u_j}$ for $j=1,\ldots,n$. We conclude that a summand appearing in
(\ref{equation:sumsexpectedvaluesummands}) equals  $0$ unless $u_j= v_j$ for
all $j=1,\ldots,n$. In other words, we have to consider only those cases where indices
$j_\alpha - k_\alpha$ and $j_{\alpha'} - k_{\alpha'+1}$ in
(\ref{equation:sumsexpectedvaluesummands}) coincide for some $\alpha,\alpha'\in\{1,\ldots,m\}$.

Combining (\ref{equation:expectedvalueinnerproducts}) and (\ref{equation:sumsexpectedvaluesummands}) we obtain
\begin{eqnarray}
  \E[{\rm Tr}H^{m}]&=&\hspace{-1cm}\sum_{\substack{\lambda_1,\ldots,\lambda_m\in\Lambda
\\ \lambda_1\neq\lambda_2\neq \lambda_3\neq\cdots\neq \lambda_m\neq\lambda_1}}
\sum_{j_1,j_2,\ldots , j_{m}=1}^n\ \ \prod_{\alpha=1}^n
    e^{2\pi i j_\alpha(\ell_{\alpha}{-}\ell_{\alpha+1})/n}  \cdot \E\left[ \prod_{\alpha=1}^n
        g(j_\alpha{-}k_\alpha) \overline{g(j_\alpha{-}k_{\alpha+1})}
    \right]. \label{equation:V2sumsexpectedvaluesummands}
\end{eqnarray}
So it remains to estimate how many of the
$|\Lambda|(|\Lambda|{-}1)^{m{-}1}(|\Lambda|{-}2)\cdot n^m$ possible combinations of
indices $\lambda_1,\ldots,\lambda_m$, $j_1,\ldots,j_m$ contribute to
(\ref{equation:V2sumsexpectedvaluesummands}) while taking into consideration that
the exponential factors in (\ref{equation:V2sumsexpectedvaluesummands}) may lead to cancelations
of nonzero summands as well.

For the sake of simple illustration we start with an example.
For given $\lambda_1,\ldots,\lambda_m$  there could exist $m$-tuples
$(j_1,\ldots,j_m)$ with
\begin{eqnarray}
    j_1 {-} k_1 &=& j_2{-}k_3 \label{equation:example-first}\\
    j_2 {-} k_2 &=& j_1{-}k_2  \label{equation:example-second}\\[.3cm]
    j_3 {-} k_3 &=& j_4{-}k_5 \\
    j_4 {-} k_4 &=& j_5{-}k_6 \\
        & \vdots & \notag \\
    j_{m{-}1} {-} k_{m{-}1}&=&j_m{-}k_1 \\
    j_m {-} k_m&=&j_3{-}k_4\,. \label{equation:example-last}
\end{eqnarray}
This scenario yields
$$
g(j_1{-}k_1)\overline{g(j_2{-}k_3)}=\frac 1 n, \ g(j_2{-}k_2)\overline{g(j_1{-}k_2)}=\frac 1 n,
\ \ldots \ ,\, g(j_m{-}k_m)\overline{g(j_3{-}k_4)}=\frac 1 n
$$
and
\begin{align}
 & \hspace{-.5cm}\E\left[
        g(j_1{-}k_1) \overline{g(j_1{-}k_2)} \,
        g(j_2{-}k_2) \overline{g(j_2{-}k_3)}
\ldots
g(j_{m{-}1}{-}k_{m{-}1}) \overline{g(j_{m{-}1}{-}k_{m})}\,
         g(j_{m}{-}k_m) \overline{g(j_{m}{-}k_{1})}
    \right]\notag\\ 
&=\,n^{-m}.\label{equation:example1}
\end{align}
Adding equations
(\ref{equation:example-first}) and (\ref{equation:example-second}) above shows that
this case, denoted in short by
\begin{eqnarray}
  1 \rightarrow 2 \rightarrow 1,
            \qquad
         3 \rightarrow  4 \rightarrow  5 \rightarrow \ldots \rightarrow  {m{-}1} \rightarrow  m \rightarrow  3, \label{equation:cyclesexample}
\end{eqnarray}
is only possible if $k_1=k_3$. Further, if this was the case, then we observe that there exists for each  $j_1=1,\ldots,n$ and $j_3=1,\ldots,n$ exactly one choice of $(m{-}2)$-tuple $(j_2,j_4,j_5,\ldots,j_m)$ satisfying
equations \eqref{equation:example-first}---\eqref{equation:example-last}, thereby implying that \eqref{equation:example1} holds.
But even these $n^2$ nonzero summands might cancel due to the phase factors present in
\eqref{equation:sumsexpectedvaluesummands}, respectively \eqref{equation:V2sumsexpectedvaluesummands}. In fact, assuming that $k_1=k_3$ holds and
that the $(m{-}2)$-tuple $(j_2,j_4,j_5,\ldots,j_m)$ is chosen to
satisfy  \eqref{equation:example-first}--\eqref{equation:example-last},
then \eqref{equation:sumsexpectedvaluesummands} becomes
\begin{eqnarray*}
  E_{\lambda_1,\ldots,\lambda_m}&=& n^{-m}
    \sum_{j_1} \sum_{j_3}
        e^{2\pi i j_1(\ell_1{-}\ell_2)/n} e^{2\pi i (j_1-k_1+k_3)(\ell_2{-}\ell_3)/n} \\[-.5cm]
        && \qquad\qquad \qquad \newtimes \ e^{2\pi i j_3(\ell_3{-}\ell_4)/n}e^{2\pi i (j_3-k_3+k_5)(\ell_4{-}\ell_5)/n}
        \cdots e^{2\pi i (j_3-k_4+k_m)(\ell_m{-}\ell_1)/n} \\[.2cm]
        &=& c_{\lambda_1,\hdots,\lambda_m} n^{-m}\Big( \sum_{j_1}
        e^{2\pi i j_1(\ell_1{-}\ell_2{+}\ell_2{-}\ell_3)/n}\Big)\Big( \sum_{j_3} e^{2\pi i j_3(\ell_3{-}\ell_4{+}\ell_4{-}\ell_5{+}\cdots {+} \ell_m{-}\ell_1)/n}\Big)\\
        &=&  c_{\lambda_1,\hdots,\lambda_m} n^{-m}\Big( \sum_{j_1}
        e^{2\pi i j_1(\ell_1{-}\ell_3)/n}\Big)\Big( \sum_{j_3} e^{2\pi i j_3(\ell_3{-}\ell_1)/n}\Big)\,,
\end{eqnarray*}
where $|c_{\lambda_1,\hdots,\lambda_m}| = 1$.
Recalling that $\sum_{j=1}^{n} e^{2\pi i j \ell/n}=0$ whenever $\ell\neq 0$, we see that the contributions in (\ref{equation:example1}) cancel out unless
$\ell_1=\ell_3$.
In short, $E_{\lambda_1,\ldots,\lambda_m}$ only contributes if $\ell_1=\ell_3$
in addition to $k_1=k_3$. 
We conclude that $|E_{\lambda_1,\hdots,\lambda_m}| = n^{-m} n^2$ if
$\lambda_1 = \lambda_3$ and $E_{\lambda_1,\hdots,\lambda_m} = 0$ otherwise.

We will now generalize the consideration of the above example in order
to derive the general estimate \eqref{equation:LemmaCount1}.

\noindent {\it Step 1.} Fix $\lambda_1,\ldots, \lambda_m$.
For $E_{\lambda_1,\ldots, \lambda_m}$ in \eqref{equation:sumsexpectedvaluesummands} to be nonzero, we must have that
   \begin{eqnarray}
  \E\left[
        g(j_1{-}k_1) \overline{g(j_1{-}k_2)} \,
        g(j_2{-}k_2) \overline{g(j_2{-}k_3)}
        \ldots
         g(j_{m}{-}k_m)  \overline{g(j_{m}{-}k_{1})}
    \right] \neq 0\label{equation:ExpNotZero}
\end{eqnarray}
for some $j_1,\ldots,j_m$.
We observed earlier, that this is only possible if each $g(j)$
in \eqref{equation:ExpNotZero}  can be paired with some $\overline{g(j)}$,
so that $\E[g(j)\overline{g(j)}]=\E[|g(j)|^2]=\tfrac 1 n$ becomes effective.
For this to be the case, the indices $1,\ldots,m$ must decompose into $s$ cycles
\begin{equation}\label{def:cycle}
         {\alpha_{11}} \rightarrow   {\alpha_{12}}\rightarrow \ldots  \rightarrow {\alpha_{1r_1}}\rightarrow   {\alpha_{11} },
            \quad \ldots \quad,\,  {\alpha_{s1}}   \rightarrow {\alpha_{s2}}\rightarrow \ldots  \rightarrow  {\alpha_{sr_s}}\rightarrow  {\alpha_{s1}} ,
\end{equation}
$r_1+r_2+\ldots+r_s=m$, where, similarly to \eqref{equation:example-first}--\eqref{equation:example-last},
\begin{eqnarray}
     j_{\alpha_{q1}} {-} k_{\alpha_{q1}} &=& j_{\alpha_{q2}}{-}k_{\alpha_{q2}+1}  \notag \\
     j_{\alpha_{q2}} {-} k_{\alpha_{q2}} &=& j_{\alpha_{q3}}{-}k_{\alpha_{q3}+1}  \notag \\[-.3cm]
                & \vdots & \label{equation:permutationrules} \\[-.3cm]
     j_{\alpha_{q(r_q{-}1)}} {-} k_{\alpha_{q(r_q{-}1)}} &=& j_{\alpha_{qr_q}}{-}k_{\alpha_{qr_q}+1}  \notag \\
     j_{\alpha_{qr_q}} {-} k_{\alpha_{qr_q}} &=& j_{\alpha_{q1}}{-}k_{\alpha_{q1}+1}  \notag
\end{eqnarray}
holds for $q=1,\ldots,s$.
Further, \eqref{equation:permutationrules} implies that whenever the $s$ equations
\begin{eqnarray}
     k_{\alpha_{q1}}+ k_{\alpha_{q2}}+\ldots +k_{\alpha_{q(r_q{-}1)}}+k_{\alpha_{qr_q}}
    &=&k_{\alpha_{q2}+1}+k_{\alpha_{q3}+1} +\ldots+k_{\alpha_{qr_q}+1}+ k_{\alpha_{q1}+1}, 
\label{equation:equationsForK}
\end{eqnarray}
$q=1,\ldots,s$,
are satisfied, then any $s$-tuple $(j_{11},j_{21},\ldots,j_{s1})\in\Z_n^{s}$ defines a nonzero value for
\begin{eqnarray}
  \E\left[
        g(j_1{-}k_1) \overline{g(j_1{-}k_2)} \,
        g(j_2{-}k_2) \overline{g(j_2{-}k_3)}
        \ldots
         g(j_{m}{-}k_m)  \overline{g(j_{m}{-}k_{1})}
    \right] \neq 0. \label{equation:EproductNEQ0}
\end{eqnarray}
Still, as we saw earlier, the contributions of summands of the
form \eqref{equation:EproductNEQ0} may cancel each other due to the phase factors
in (\ref{equation:expectedvalueinnerproducts}). In fact,
for $j_{\alpha_{qp}}$, $q=1,\ldots,s$ and $p=1,\ldots,r_q$ satisfying
\eqref{equation:permutationrules}, we have
\begin{eqnarray*}
  E_{\lambda_1,\ldots,\lambda_m}&=&\sum_{j_1,\ldots,j_m=1}^n \E\left[\prod_{\alpha=1}^m
            e^{2\pi i j_{\alpha}\ell_{\alpha}}  e^{-2\pi i j_{\alpha}\ell_{\alpha{+}1}}
g(j_\alpha{-}k_\alpha)\, \overline{g(j_\alpha{-}k_{\alpha{+}1})}\right]\\
&=&\sum_{j_{11},j_{21},\ldots,j_{s1}=1}^n\quad
  n^{-m}  \prod_{q=1}^s \prod_{p=1}^{r_q}
            e^{2\pi i \big(j_{\alpha_{qp}}\ell_{\alpha_{qp}}- j_{\alpha_{qp}}\ell_{\alpha_{qp}{+}1}\big)}  \\
&=&\sum_{j_{11},j_{21},\ldots,j_{s1}=1}^n
   n^{-m} \prod_{q=1}^s  \prod_{p=1}^{r_q}
            e^{2\pi i j_{\alpha_{qp}}\big(\ell_{\alpha_{qp}}- \ell_{\alpha_{qp}{+}1}\big)}\\
&=&\sum_{j_{11},j_{21},\ldots,j_{s1}=1}^n
  n^{-m}   \prod_{q=1}^s \prod_{p=1}^{r_q}
            e^{2\pi i \big(j_{\alpha_{q1}}+\widetilde{k}_{\alpha_{qp}}\big)\big(\ell_{\alpha_{qp}}- \ell_{\alpha_{qp}{+}1}\big)} \\
&=&n^{-m} e^{2\pi i  \sum_{q=1}^s\sum_{p=1}^{r_q}\widetilde{k}_{\alpha_{qp}}\big(\ell_{\alpha_{qp}}- \ell_{\alpha_{qp}{+}1}\big)}\sum_{j_{11},j_{21},\ldots,j_{s1}=1}^n\quad
    \prod_{q=1}^s e^{2\pi i j_{\alpha_{q1}}\big(\sum_{p=1}^{r_q}\ell_{\alpha_{qp}}- \ell_{\alpha_{qp}{+}1}\big)}\\
&=& n^{-m}c_{\lambda_1,\ldots,\lambda_m} \left(\sum_{j_{11}=1}^n e^{2\pi i j_{\alpha_{11}}\big(\sum_{p=1}^{r_1}\ell_{\alpha_{1p}}- \ell_{\alpha_{1p}{+}1}\big)}\right)\ldots
\left(\sum_{j_{s1}=1}^n e^{2\pi i j_{\alpha_{s1}}\big(\sum_{p=1}^{r_1}\ell_{\alpha_{sp}}- \ell_{\alpha_{sp}{+}1}\big)}\right)
\end{eqnarray*}
with $|c_{\lambda_1,\ldots,\lambda_m}| = 1$
and $\widetilde{k_{\alpha_{q1}}}=0$, $\widetilde{k_{\alpha_{q2}}}=k_{\alpha_{q2}}-k_{\alpha_{q1}}$,  $\widetilde{k_{\alpha_{q3}}}=k_{\alpha_{q3}}+k_{\alpha_{q2}}-k_{\alpha_{q1}}$, etc.
Hence, $E_{\lambda_1,\ldots,\lambda_m}=0$ if not simultaneously, we have
\begin{eqnarray*}
 \ell_{\alpha_{q1}}+ \ell_{\alpha_{q2}}+\ldots +\ell_{\alpha_{q(r_1{-}1)}}+\ell_{\alpha_{qr_1}}
 &=&\ell_{\alpha_{q2}+1}+\ell_{\alpha_{q3}+1} +\ldots+\ell_{\alpha_{qr_q}+1}+ \ell_{\alpha_{q1}+1}
\end{eqnarray*}
for $q=1,\ldots, s$ in which case $E_{\lambda_1,\ldots, \lambda_m}=c_{\lambda_1,\ldots, \lambda_m}n^{s-m}$. Consequently $|E_{\lambda_1,\hdots,\lambda_m}| = n^{s-m}$ whenever $\lambda_1,\ldots,\lambda_m$ satisfies the $s$ linear equations
\begin{eqnarray}
     \lambda_{\alpha_{q1}}+ \lambda_{\alpha_{q2}}+\ldots +\lambda_{\alpha{q(r_q{-}1)}}+\lambda_{\alpha_{qr_q}}
    &=&\lambda_{\alpha_{q2}+1}+\lambda_{\alpha_{q3}+1} +\ldots+\lambda_{\alpha_{qr_q}+1}+ \lambda_{\alpha_{q1}+1}
\label{equation:systemofequations},
\end{eqnarray}
$q=1,\ldots,s$.
We conclude that of the $|\Lambda|(|\Lambda|{-}1)^{m{-}1}(|\Lambda|{-}2)\cdot n^m$ summands in \eqref{equation:V2sumsexpectedvaluesummands}, only those need
to be considered that correspond to a partition of the indices
$1,\ldots,m$ of the $j$'s into cyclic permutations and where the
$\lambda_1,\ldots,\lambda_m$ satisfy a corresponding system (\ref{equation:systemofequations}) of equations.
This observation will be used to estimate $\E[\Tr H^m]$  in the following step.

\noindent {\it Step 2.}
We observe in addition to the above, that $E_{\lambda_1,\hdots,\lambda_m}$ with $\lambda_i = (k_i,\ell_i)$
contributes only if
$k_1\neq k_2$, $k_2 \neq k_3$, $\hdots$, $k_m \neq k_1$. Indeed,
$k_i = k_{i+1}=k$ implies $\ell_i \neq \ell_{i+1}$ by $\lambda_i \neq \lambda_{i+1}$.
But since $g$ and hence $T_{k} g$ is unimodular by assumption, the set
$\{M_\ell T_kg , \ell = 0,\hdots,n-1\}$ forms an orthogonal system, and we have
$\langle M_{\ell_i} T_k g, M_{\ell_{i+1}} T_k g\rangle = 0$, implying that
$E_{\lambda_1,\hdots,\lambda_m} = 0$. The condition $k_i \neq k_{i+1}$ in turn
implies that each cycle in (\ref{def:cycle}) has at least two elements, as
otherwise, (\ref{equation:permutationrules}) would lead to a contradiction.

Now for each permutation with $s$ cycles
described by (\ref{def:cycle}) we give an upper bound on the
number of index tuples
$(\lambda_1,\hdots,\lambda_m)$ satisfying the $s$ equations
(\ref{equation:systemofequations}). To this end, we shall show that any $s-1$ equations of the $s$ equations
(\ref{equation:systemofequations}) are linearly independent. 

First, note that each $\lambda_j$ appears on exactly one left hand side 
and on one right hand side of the equations (\ref{equation:systemofequations}).
Hence, a linear combination of these equations leading to the trivial
equation $0=0$ can be achieved involving only $0$'s 
and $1$'s as coefficients, that is, by simply adding up some of the equations 
in (\ref{equation:systemofequations}). But the fact that the right hand side 
of an equation consists exactly of the successor variables of the left hand 
side implies that a vanishing sum of equations in 
(\ref{equation:systemofequations}) must contain all variables on both sides. 
As all equations are non-trivial, this is achieved if and only if the sum is 
taken over all equations. Hence, the $s$ equations 
(\ref{equation:systemofequations}) are linearly dependent, 
while any $s-1$ equations are not.
We conclude that the system
(\ref{equation:systemofequations}) describes an $m-(s-1)$-dimensional
subspace
whose its intersection
with $\Lambda^m$ has at most $|\Lambda|^{m-(s-1)}$ elements.

By definition of the associated Stirling numbers of the first kind
there are $d_2(m,s)$ permutations with $s$ disjoint cycles of
minimum length 2 of the index set $\{1,\ldots,m\}$. Each of these permutations
represent $n^{s}$ tuples $(j_1,\ldots,j_m)$ and,
at most $|\Lambda|^{m{-}(s{-}1)}$ tuples $(\lambda_1,\ldots,\lambda_m) \in \Lambda^m$
satisfing
\eqref{equation:systemofequations}. Each of these tuples of indices gives a
contribution to (\ref{equation:V2sumsexpectedvaluesummands})
of absolute value at most $n^{-m}$. Finally, this yields
\begin{eqnarray}
  \E[{\Tr}H^{m}]&\leq& \sum_{s=1}^{m/2} d_2(m,s) \ |\Lambda|^{m-(s{-}1)}  \ n^{s-m} \notag\\
    &=&|\Lambda| \left(\frac{|\Lambda|} n \right)^{m} \ \ \sum_{s=1}^{m/2} d_2(m,s) \,\left(\frac n {|\Lambda|} \right)^s\notag
\end{eqnarray}
and the proof of Lemma~\ref{lem:frob} is complete.
\end{proof}

\subsection{Proof of Theorem \ref{thm:opnorm}}

Given specific parameters $n$ and $\sparsity= |\Lambda|$ one may already obtain
good estimates for the probability that $\| I_\Lambda - \Psi_\Lambda^* \Psi_\Lambda\|\leq \delta$ by numerically minimizing
the right hand side of \eqref{equation:LemmaCount1} over $m \in \N$ and using
(\ref{Pineq1}).
Theorem~\ref{thm:opnorm} is proven
by pursuing a similar strategy combined with an estimate of the
numbers $d_2(m,s)$, compare also to \cite{grpora07}.

We first claim that the associated Stirling numbers
of the first kind satisfy the estimate
\begin{eqnarray}
 d_2(m+1,s) \,\leq\, (2m)^{m-s}. \label{equation:d2mEstimate}
\end{eqnarray}
Indeed, \eqref{equation:d2mEstimate} is true for $m\geq 1$ and
$s=0$ or $s > m/2$ since then $d_2(m,s) = 0$ by (\ref{bound:cond}).
It is also true for $d_2(2,1) = 1$.
Now let $m\geq 2$ and suppose the claim is true for all
$d_2(m',s)$ with $m' \leq m$
and $s\geq 0$.
Then
\begin{align}
d_2(m+1,s) &= m(d_2(m,s) + d_2(m-1,s{-}1))
\leq m ((2(m-1))^{m-1-s} + (2(m-2))^{m-2-(s{-}1)})\notag\\
&\leq 2m (2m)^{m-1-s} = (2m)^{m-s}. \notag
\end{align}
Now let
\begin{equation}\label{def:Gm}
G_{2m}(z) \,:=\,z^{-2m} \sum_{s=1}^m d_2(2m,s) z^s.
\end{equation}
By Lemma \ref{lem:frob},
$\E [\|H^{m}\|^2] \leq \E [\Tr H^{2m}] \leq \sparsity G_{2m}(n/\sparsity)$.
Using the estimate \eqref{equation:d2mEstimate} we obtain
\begin{align}
G_{2m}(z) \, &\leq \, z^{-2m} \sum_{s=1}^m (2 (2m-1))^{2m-1-s} z^s
\leq \left(\frac{4m}{z}\right)^{2m} (4m)^{-1} \sum_{s=1}^m (z/4m)^{s}\notag\\
&=\, \left(\frac{4m}{z}\right)^{2m} (4m)^{-1} \frac{(z/4m)^{m+1} - (z/4m)}{z/4m -1}\notag\\
&= \, (4m)^{-1} \left(\frac{4m}{z}\right)^{m} \frac{1 - (4m/z)^m}{1- (4m/z)}.
\end{align}
Now choose $m = m_z \in \N$ such that $4m_z/z \leq \alpha < 1$, for instance
\begin{align}\label{def:mx}
m_z \,:=\, \left\lfloor \frac{\alpha z}{4} \right\rfloor.
\end{align}
Then
\begin{equation}
G_{2m_z}(z) \leq (4m_z)^{-1} \frac{\alpha^{m_z}}{1- \alpha} \leq
\frac{\alpha^{m_z}}{4(1-\alpha)}. \label{G:estim}
\end{equation}
We want to achieve $\P\big(\|H\| > \delta\big) \leq \varepsilon$, which by (\ref{Pineq1})
will be satisfied provided
\[
\delta^{-2m_z} \sparsity \frac{\alpha^{m_z}}{4(1-\alpha)} \leq \varepsilon
\]
for $z = n/\sparsity$.
Assuming $\alpha < \delta^2$ the latter inequality is equivalent to
\[
m_z \log(\delta^2/\alpha)\, \geq \, \log\left(\frac{\sparsity}{4\varepsilon(1-\alpha)}\right).
\]
Plugging in $z = n/\sparsity$ and $m_z$ given by (\ref{def:mx}) we obtain
\[
\left\lfloor \frac{\alpha n}{4\sparsity} \right\rfloor \log\left(\delta^2/\alpha\right) \geq
\log \left( \frac{\sparsity}{4\varepsilon (1-\alpha)}\right).
\]
Finally, choose $\alpha = \delta^2 / e$. Then the above inequality reduces to
\[
\left\lfloor \frac{\delta^2 n}{4e\sparsity} \right\rfloor \geq
\log \left( \frac{\sparsity}{4 \varepsilon (1- \delta^2/e)}\right).
\]
A straightforward calculation shows that the above
equation is satisfied whenever
\[
n \geq \frac{4e}{\delta^2} \sparsity \log\left(\frac{e}{4(1-\delta^2/e)}\frac{\sparsity}{\varepsilon}\right).
\]
Finally, the above inequality is implied by the assumption
of Theorem~\ref{thm:opnorm}.

\begin{remark} Starting from the first inequality in (\ref{G:estim}) and
proceeding analogously as in the previous proof one may
deduce the slightly better but more complicated condition
\[
n \geq \frac{4e}{\delta^2} \sparsity \left(\log\left(\frac{\sparsity^2}{\delta^2 n - 4 e \sparsity}\, \varepsilon^{-1} \right) + \log\left(\frac{e^3}{e-1}\right) \right).
\]
ensuring $\|I_\Lambda - \Psi_\Lambda^* \Psi_\Lambda\|\leq \delta$ with probability
at least $1-\varepsilon$.
\end{remark}

\section{Recovery of random signals}
\label{Sec:RPh}

Theorem~\ref{thm:randomphase} addresses the recovery of signals
whose sparse coefficients in a Gabor expansion are chosen with random phases.
Its proof is based on a recovery result due
to Tropp \cite{tr05-1} and Fuchs \cite{fu04} which is given in our framework as Lemma~\ref{lem:Tropp} below.

Let $\psi_\lambda = \pi(\lambda) g$ be the column of $\Psi$ indexed by $\lambda$.
By $R_\Lambda x$ we denote the restriction of a vector $x$ to the
index set $\Lambda$. Furthermore, $\sgn(x)$ is the sign of a vector, that is,
$\sgn(x)_k = x_k/|x_k|$ for the non-zero entries of $x$ and $\sgn(x)_k =0$ else.

\begin{lemma}\label{lem:Tropp} Suppose that $y = \Psi x$ for some $x$ with $\supp x = \Lambda$.
If
\begin{equation}\label{recovery:cond}
|\langle \Psi_\Lambda^\dagger \psi_\rho, R_\Lambda \sgn(x) \rangle | < 1 \quad \mbox{ for all } \rho \notin \Lambda\,,
\end{equation}
then $x$ is the unique solution of the Basis Pursuit problem (\ref{def:Basis Pursuit}).
Here $\Psi_\Lambda^\dagger$ denotes the Moore-Penrose
pseudo-inverse of $\Psi_\Lambda$.
\end{lemma}

\subsection{Proof of Theorem \ref{thm:randomphase}}

We will use Lemma \ref{lem:Tropp} in combination
with Theorem~\ref{thm:opnorm} and an estimation of the coherence
of $\Psi$ given in \cite{pfrata07} to prove Theorem~\ref{thm:randomphase}
concerning recovery by Basis Pursuit of sparse signals with
random phases.

We aim at using the following Bernstein type inequality
for a sequence of independent random variables $\epsilon_k$ having
uniform distribution on the torus \cite[Proposition 16]{tr06-2},
\begin{equation}\label{Bernstein}
\P\big(|\sum_{j} \epsilon_j a_j|\geq u \|a\|_2\big) \leq \frac{e^{-\kappa u^2}}{1-\kappa}
\end{equation}
for any $\kappa \in (0,1)$. By our assumption on the random
phases $\epsilon_\lambda = \sgn(x_\lambda)$, the scalar product on the
left hand side of (\ref{recovery:cond}) is precisely of the above form
with $a = \Psi_\Lambda^\dagger \psi_\rho$. The $2$-norm of this particular $a$
can be estimated by
\[
\|\Psi_\Lambda^\dagger \psi_\rho\|_2 \,=\, \|(\Psi_\Lambda^* \Psi_\Lambda)^{-1} \Psi_\Lambda^* \psi_\rho \|_2 \leq \|(\Psi_\Lambda^* \Psi_\Lambda)^{-1}\|
\|\Psi_\Lambda^* \psi_\rho\|_2.
\]
Now suppose that $\|I - \Psi_\Lambda^* \Psi_\Lambda\|
\leq \delta$. The probability
that this is the case is estimated by Theorem~\ref{thm:opnorm}. Then
$\|(\Psi_\Lambda^* \Psi_\Lambda)^{-1}\|
\leq \frac{1}{1-\delta}$. Furthermore,
observe that
\[
\|\Psi_\Lambda^* \psi_\lambda\|_2 =
\left( \sum_{\lambda \in \Lambda} |\langle \psi_{\lambda}, \psi_\rho \rangle|^2 \right)^{1/2}
\leq \sqrt{\sparsity} \mu,
\]
where $\mu = \max_{\lambda'\neq \lambda} |\langle \psi_{\lambda'},\psi_\lambda \rangle|$
denotes the coherence of $\Psi$. Combining the above estimates yields
\begin{equation}\label{equation:CombineEstimates1}
\|\Psi_\Lambda^\dagger \psi_\rho\|_2 \leq \frac{1}{1-\delta} \sqrt{\sparsity} \mu.
\end{equation}
Now, the probability that recovery fails can be estimated by
\begin{align}
&\P\big(|\langle \Psi_\Lambda^\dagger \psi_\rho, R_\Lambda \sgn(x) \rangle| \geq 1
\mbox{ for some } \rho \notin \Lambda\big)\notag\\[.2cm]
&\leq \P\big(\,|\langle \Psi_\Lambda^\dagger \psi_\rho, R_\Lambda \sgn(x) \rangle| \geq 1
\mbox{ for some } \rho \notin \Lambda\ \big| \ \mu \leq \frac{\alpha}{\sqrt{n}}\, \mbox{\&}\, \|H\| \leq \delta\,\big) \notag\\[.2cm]&\qquad + \P\big(\mu > \frac{\alpha}{\sqrt{n}}\big) + \P\big(\|H\| > \delta\big) \notag\\[.2cm]
&\leq \sum_{\rho \notin \Lambda} \P\big(\,|\langle \Psi_\Lambda^\dagger \psi_\rho, R_\Lambda \sgn(x) \rangle| \geq 1\ \big| \ \mu \leq \frac{\alpha}{\sqrt{n}} \, \mbox{\&}\, \|H\| \leq \delta\,\big) \notag\\[.2cm]&\qquad+
\P\big(\mu > \frac{\alpha}{\sqrt{n}}\big) + \P\big(\|H\| > \delta\big).\notag
\end{align}
Equation \eqref{equation:CombineEstimates1} implies that for $u=\frac {(1-\delta)\sqrt n}{\alpha \sqrt S}$ we have $u\| \Psi_\Lambda^\dagger \psi_\rho\|_2\leq 1$, so (\ref{Bernstein}) gives
\begin{equation}\label{equation:CombineEstimates2}
\P\big(\,|\langle \Psi_\Lambda^\dagger \psi_\rho, R_\Lambda \sgn(x) \rangle| \geq 1\ \big|\ \mu \leq \frac{\alpha}{\sqrt{n}}\,  \mbox{\&} \, \|H\| \leq \delta\, \big)\
\leq \ (1-\kappa)^{-1} \exp\left(-\kappa \frac{(1-\delta)^2}{\alpha^2} \frac{n}{\sparsity}\right).
\end{equation}
In \cite[Theorem 5.1]{pfrata07} it was proven that
\begin{equation}\label{coherence:gR}
\P\big(\mu > \frac{\alpha}{\sqrt{n}}\big) \leq 2(1-\kappa')^{-1}n (n-1) \exp(-\kappa'\alpha^2/2)
\end{equation}
for any $\kappa' \in (0,1)$. Combining  inequalities \eqref{equation:CombineEstimates2}, \eqref{coherence:gR} and \eqref{ineq:op2}, we obtain the following bound on the probability that recovery
fails
\[
(n^2-\sparsity) (1-\kappa)^{-1} \exp\left(-\kappa\frac{(1-\delta)^2}{\alpha^2} \frac{n}{\sparsity}\right)
+ 2(1-\kappa')^{-1}n (n-1) \exp(-\kappa'\alpha^2/2) + C \sparsity \exp\left(-\frac{\delta^2 n}{4e \sparsity}\right).
\]
Choosing $\alpha = \sqrt{\log(n^\sigma)}$ with $\sigma > 8$,
$\kappa = \kappa' = 1/2$ and
$\delta = 1/2$ the above expression equals
\begin{align}
&2(n^2-\sparsity) \exp\left(-\frac{1}{8\sigma} \frac{n}{\sparsity \log(n)}\right)
+4n(n-1) n^{-\sigma/4} + C\sparsity \exp\left(-\frac{1}{16e} \frac{n}{\sparsity}\right)\notag\\[.2cm]
&\qquad  \leq 2(n^2-\sparsity) \exp\left(- \frac{n}{8\sigma \sparsity\log(n)}\right)
+ 4 n^{-(\sigma/4 - 2)} + CS \exp\left(- \frac{n}{16e \sparsity}\right) \notag.
\end{align}
This completes the proof.

\section{Recovery of deterministic signals}
\label{Sec:Det}

In this section we prove Theorem~\ref{thm:recover}. As
an auxiliary tool we first provide a general recovery lemma.

\subsection{A general recovery lemma}

The following lemma holds for any (random) matrix $\Psi$.
It is inspired by the analysis performed in \cite{carota06} and \cite{ra05-7}.

Let $\Lambda$ be a subset of the column index set of $\Psi$,
and $\Lambda^c$ its complement.
Let $E_\Lambda = R_\Lambda^*$ be the adjoint of the restriction
operator $R_\Lambda$; clearly, $E_\Lambda$
extends a vector outside $\Lambda$ by $0$. Further, we define
\[
H = \Psi_\Lambda^* \Psi_\Lambda - I_\Lambda
\qquad \mbox{and} \qquad
K = \Psi^* \Psi_\Lambda - E_\Lambda.
\]
Observe that $H = R_\Lambda K$. With this notation we have
\begin{lemma}\label{lem:aux} Let $x$ be supported on $\Lambda$ with $|\Lambda| = \sparsity$.
Let $\beta > 0$, $\kappa > 0$,
$m \in \N$ and $L_t \in \N, t=1,\hdots,m$, be parameters such that
\begin{equation}\label{cond_kappa00}
a := \sum_{t=1}^m \beta^{m/L_t} < 1 \qquad\mbox{and}\qquad
\frac{\kappa}{1-\kappa} \leq \frac{1-a}{1+a}\, \sparsity^{-3/2}.
\end{equation}
Then with probability at most
\begin{equation}\label{prob:est}
 \kappa^{-2} \E\left[\|H^m\|_F^2\right] + \beta^{-2m}\sum_{\rho \in \Lambda^c} \sum_{t=1}^m \E\left[|((KR_\Lambda)^t \sgn(x))_\rho|^{2L_t} \right]
\end{equation}
Basis Pursuit fails to recover $x$ from $\Psi x$.
\end{lemma}
\begin{proof} We reassemble the arguments
from \cite{ra05-7,carota06} for the reader's convenience.

First, we address the recovery condition (\ref{recovery:cond}).
Let $\Lambda$ be
the support of $x$. Define
\[
P \,:=\, \Psi^* \Psi_\Lambda (\Psi_\Lambda^* \Psi_\Lambda)^{-1} R_\Lambda \sgn(x).
\]
Note that condition (\ref{recovery:cond}) in Lemma \ref{lem:Tropp}
is equivalent to
\[
\| R_{\Lambda^c} P \|_\infty < 1.
\]
The vector $P$ can also be written as
\[
P \,=\, (E_\Lambda + K) \left(I_\Lambda + H\right)^{-1} R_\Lambda \sgn(x),
\]
and, since $R_{\Lambda^c} E_\Lambda = 0$ we have
\[
R_{\Lambda^c}P \,=\, R_{\Lambda^c} K (I_\Lambda + H)^{-1} R_\Lambda \sgn(x).
\]
Let us look closer at the term $(I_\Lambda + H)^{-1}$.
By the Neumann
series we can write
$
\left(I_\Lambda - (-H)^m\right)^{-1} = I_\Lambda + A_m
$
with
\begin{equation}\label{def_An}
A_m \,:=\, \sum_{r=1}^\infty (-H)^{rm}.
\end{equation}
Using the algebraic identity
$(1-M)^{-1} \,=\, (1- M^m)^{-1}(1 + M + \cdots + M^{m{-}1})$
we obtain
\[
(I_\Lambda + H)^{-1} \,=\, (I_\Lambda +A_m)\sum_{t=0}^{m{-}1} (-H)^t.
\]
Thus, on the complement of $\Lambda$, we may write
\begin{eqnarray*}
R_{\Lambda^c} P &=&  K (I_\Lambda + A_m)
\left(\sum_{t=0}^{m{-}1} (-H)^t\right) R_\Lambda \sgn(x)\\
&=& P^{(1)} + P^{(2)},
\end{eqnarray*}
where
\[
P^{(1)} \,:=\, -Q_m \, \sgn(x) \quad\mbox{ and }\quad
P^{(2)} \,:=\, K A_m R_\Lambda (I + Q_{m{-}1})\, \sgn(x),
\]
with
\[
Q_m \,:=\,  -\sum_{t=0}^{m-1} K (- R_\Lambda K)^t R_\Lambda  \,= \, \sum_{t=1}^m (-K R_\Lambda)^t.
\]
With this at hand, we can now proceed to estimate $\P\big(\sup_{\rho \in \Lambda^c} |P_\rho| \geq 1\big)$.
To this end let $a,b>0$ be numbers satisfying $a + b = 1$. Then
\begin{equation}\label{Psplit}
\P\big(\sup_{\rho \in \Lambda^c} |P_\rho| \geq 1\big) \,\leq\,
\P\left(\{\sup_{\rho \in \Lambda^c} |P^{(1)}_\rho| \geq a\} \cup
\{\sup_{\rho \in \Lambda^c} |P^{(2)}_\rho| \geq b\}\right).
\end{equation}
Clearly,
\begin{align}
\P\big(|P^{(1)}_\rho|\geq a\big)
\label{Pinf_estim}
\, \leq \, \P\left(\sum_{t=1}^m |((KR_\Lambda)^t \sgn(x))_\rho| \geq a\right)
\,=:\, \P\big(\Omega_\rho\big), \quad \rho \in \Lambda^c.
\end{align}
For $P^{(2)}$ we obtain
\begin{equation}\label{P2_ineq}
\sup_{\rho \in \Lambda^c} |P^{(2)}_\rho| \leq \|P^{(2)}\|_\infty
\leq \|K A_m\|_{\infty \to \infty}
(1 + \|R_\Lambda Q_{m{-}1} \sgn(x)\|_{\infty}).
\end{equation}
In order to analyze the term $\|R_\Lambda Q_{m{-}1} \sgn(x)\|_{\infty}$
we observe that similarly as in (\ref{Pinf_estim})
\[
\P\big(|(Q_{m{-}1} \sgn(x))_\rho| \geq a\big)\,\leq\,
\P\left(\sum_{t=1}^m |(K R_\Lambda)^t \sgn(x))_\rho| \geq a\right) \,=\,\P\big(\Omega_\rho\big), \quad \rho \in \Lambda^c.
\]
Let us now focus on the operator norm appearing in (\ref{P2_ineq}).
It holds $\|A\|_{\infty \to \infty} = \sup_{j} \sum_\ell |A_{j\ell}|$. Clearly,
\begin{equation}\label{equation:CombineEstimates3}
\|K A_m\|_{\infty\to \infty} \leq \|K\|_{\infty \to \infty} \|A_m\|_{\infty \to \infty}.
\end{equation}
Moreover,
\begin{equation}\label{equation:CombineEstimates4}
\|K\|_{\infty \to \infty} \leq |\Lambda| = \sparsity
\end{equation}
as $K$ has $\sparsity$ columns and
each entry is bounded by $1$ in absolute value.

Let us analyze $A_m$ using the Frobenius norm.
Assume for the moment that
\begin{equation}\label{assume_H0}
\|H^m\|_F \,\leq\, \kappa < 1.
\end{equation}
Then it follows directly from the definition (\ref{def_An}) of  $A_m$ that
\[
\|A_m\|_F \,=\, \left\|\sum_{r=1}^\infty (-H)^{rm}\right\|_F
\,\leq\, \sum_{r=1}^\infty \| H^m\|_F^r
\,\leq\, \sum_{r=1}^\infty \kappa^{r}
\,=\, \frac{\kappa}{1-\kappa}.
\]
Moreover, since $A_m$ has $|\Lambda| = \sparsity$ columns
it follows from the Cauchy-Schwarz inequality that
\begin{equation}\label{equation:CombineEstimates5}
\|A_m\|_{\infty\to\infty}^2 \,\leq\, \sup_{\lambda \in \Lambda} |\Lambda| \sum_{\lambda'} |(A_m)_{\lambda,\lambda'}|^2
\,\leq\, \sparsity \|A_m\|_F^2.
\end{equation}
So assuming (\ref{assume_H0})
and $\|Q_{m{-}1} \sgn(x)\|_\infty < a$, we can combine \eqref{equation:CombineEstimates3}, \eqref{equation:CombineEstimates4} and \eqref{equation:CombineEstimates5} to obtain
\[
\sup_{\rho \in \Lambda^c} |P^{(2)}_\rho| \,\leq\, (1+a) \, \sparsity^{3/2}\,
\frac{\kappa}{1-\kappa}.
\]
By assumption of the lemma
\begin{equation}
\frac{\kappa}{1-\kappa} \, \leq\,  \frac{1-a}{1+a}\,  \sparsity^{-3/2}\, =\, \frac{b}{1+a}\, \sparsity^{-3/2},
\end{equation}
and $\sup_{\rho\in \Lambda^c} |P^{(2)}_\rho| \leq b$ under condition
(\ref{assume_H0}) as desired.
Also it follows from (\ref{cond_kappa00}) that $\kappa < 1$ as
$\sparsity \geq 1$ without loss of generality
(if $\Lambda=\emptyset$ then $x=0$
and $\ell^1$-minimization will clearly recover $x$.)

Using the union bound we  obtain from (\ref{Psplit})
\begin{eqnarray}
\P\big(\sup_{\rho \in \Lambda^c} |P_\rho| \geq 1\big) &\leq&
\P\left(\bigcup_{\rho \in \Lambda^c}\{|P_\rho^{(1)}| \geq a\}
\cup \{\|Q_{m-1} \sgn(x)\|_{\infty} \geq a\}
\cup \{\|H^m\|_F \geq \kappa\}\right) \notag\\
&\leq& \P\left(\bigcup_{\rho \in \Lambda^c} \Omega_\rho \cup
\{\|H^m\|_F \geq \kappa\}\right)
\leq \sum_{\rho \in \Lambda^c} \P\big(\Omega_\rho\big) +
\P\big(\|H^m\|_F \geq \kappa\big).\label{P_estim1}
\end{eqnarray}
Markov's inequality now gives
\begin{align}
\P\big(\|H^m\|_F \geq \kappa\big) \,&=\, \P\big(\|H^m\|^2_F \geq \kappa^2\big) \,\leq\, \kappa^{-2} \E[\|H^m\|_F^2].
\label{P:frob}
\end{align}
It remains to investigate $P(\Omega_\rho)$.
To this end let $\beta_t, t=1,\hdots,m$, be positive numbers satisfying
\[
\sum_{t=1}^m \beta_t = a
\]
and let $L_t \in \N$, $t=1,\hdots,m$. For $\rho \in \Lambda^c$, we have
\begin{align}
\P\big(\Omega_\rho\big) \,& = \, \P\left(\sum_{t=1}^m |((KR_\Lambda)^t \sgn(x))_\rho| \geq a\right)
\,\leq\, \sum_{t=1}^m \P\big(|((KR_\Lambda)^t \sgn(x))_\rho| \geq \beta_t\big) \notag\\
&=\, \sum_{t=1}^m \P\left(|((KR_\Lambda)^t \sgn(x))_\rho|^{2L_t} \geq \beta_t^{2L_t}\right)
\label{P1_estim}
\,\leq \, \sum_{t=1}^m \E\left[|((KR_\Lambda)^t\sgn(x))_\rho|^{2L_t}\right]
\beta_t^{-2L_t},
\end{align}
where  Markov's inequality was used to obtain the last inequality.
Let us choose $\beta_t = \beta^{m/L_t}$, that is,
$\beta_t^{-2L_t} = \beta^{-2m}$. This yields
\begin{equation}\label{prob_Ek}
\P\big(\Omega_\rho\big) \,\leq\, \beta^{-2m} \sum_{t=1}^m
\E\left[|((KR_\Lambda)^t\sgn(x))_\rho|^{2L_t}\right]
\end{equation}
and the condition $a=\sum_{t=1}^m \beta_t$ reads
\begin{equation*}
a \,=\, \sum_{t=1}^m \beta^{m/L_t} < 1.
\end{equation*}
This is precisely the first condition in (\ref{cond_kappa00}). Assembling \eqref{P_estim1}, \eqref{P:frob} and \eqref{prob_Ek}  completes the proof.
\end{proof}

\subsection{Estimate of an auxiliary expected value}

Lemma \ref{lem:aux} suggests the investigation of the expected values appearing
in (\ref{prob:est}). As the expectation $\E\left[\|H^m\|_F^2\right]$ was already estimated
in Lemma \ref{lem:frob}, we focus here on terms of the form $\E\left[|((KR_\Lambda)^t\sgn(x))_\rho|^{2L_t}\right]$.

\begin{lemma}\label{lem:expaux} Let $\Lambda \subset \Z_n \times \Z_n$
with $|\Lambda| = \sparsity$. Then for $\rho\notin \Lambda$,
\[
\E\left[|((KR_\Lambda)^t\sgn(x))_\rho|^{2L}\right]
\leq \left(\frac{\sparsity}{n}\right)^{2tL}\sum_{s=1}^{tL} d_2(2tL,s) \left(\frac{n}{\sparsity}\right)^s.
\]
\end{lemma}
\begin{proof} Note that
$(KR_\Lambda)^t = K (R_\Lambda K)^{t-1} R_\Lambda = K H^{t-1} R_\Lambda$. Denote
$\sigma = R_\Lambda \sgn(x)$. Then for
$\rho \notin \Lambda$ we have
\[
((K R_\Lambda)^t) \sigma)_{\rho}
\,=\,
\sum_{\lambda_1\neq \lambda_2 \neq \lambda_3 \hdots \neq \lambda_t \in \Lambda}
 \langle \pi(\rho)g,\pi(\lambda_1)g \rangle\, \langle \pi(\lambda_1)g,\pi(\lambda_2)g \rangle
\cdots \langle \pi(\lambda_{t-1})g,\pi(\lambda_t)g \rangle\,\sigma(\lambda_t)\,.
\]
Furthermore, setting
$\lambda^{2u-1}_0:= \lambda^{2u}_t := \rho \notin \Lambda$, $u=1,\hdots,L$, for notational brevity, we can write
\begin{align}
|((K R_\Lambda)^t \sigma)_{\rho}|^{2L}
\,=\, \sum_{\substack{\lambda_1^1, \hdots, \lambda_t^1 \in \Lambda\\ \lambda_0^2,\hdots,\lambda_{t-1}^2\in \Lambda\\ \vdots \\ \lambda_0^{2L},\hdots,\lambda_{2t-1}^{2L}  \in \Lambda\\\lambda_{r-1}^u \neq \lambda_{r}^u}}
 \prod_{u=1}^L \sigma(\lambda_{t}^{2u-1}) \overline{\sigma(\lambda_{0}^{2u})}
\prod_{r=1}^t \langle \pi(\lambda_{r-1}^{2u-1}) g, \pi(\lambda_{r}^{2u-1}) g \rangle \langle \pi(\lambda_{r-1}^{2u}) g, \pi(\lambda_{r}^{2u})g \rangle.\notag
\end{align}
Using linearity of expectation we obtain
\begin{align}
\E |((K R_\Lambda)^t \sigma)_{\rho}|^{2L}
\,=\,  \sum_{\substack{\lambda_1^1, \hdots, \lambda_t^1 \in \Lambda\\ \vdots\\\lambda_0^{2L},\hdots,\lambda_{t-1}^{2L}  \in \Lambda\\\lambda_{r-1}^u \neq \lambda_{r}^u}}
S_{\lambda_1^1,\hdots,\lambda_t^{2L}} F_{\lambda_1^1,\hdots,\lambda_{t}^{2L}}
\end{align}
where $S_{\lambda_1^1,\hdots,\lambda_t^{2L}}$ does not depend on $g$, $\left|S_{\lambda_1^1,\hdots,\lambda_t^{2L}}\right| = 1$ and
\[
F_{\lambda_1^1,\hdots,\lambda_{t}^{2L}} =
\E\left[  \prod_{u=1}^L \prod_{r=1}^t \langle \pi(\lambda_{r-1}^{2u-1}) g, \pi(\lambda_{r}^{2u-1}) g \rangle \langle \pi(\lambda_{r-1}^{2u}) g, \pi(\lambda_{r}^{2u})g \rangle\right].
\]
Let us write $\lambda_r^u = (k_r^u,\ell_r^u)$. If $k_{r-1}^u = k_{r}^u$ for $r \in \{1,\hdots,t\}, u \in \{1,\hdots, 2L\}$ then
necessarily $\ell_{r-1}^u \neq \ell_{r}^u$ due to the condition
$\lambda_{r-1}^{u} \neq \lambda_{r}^u$ on the index set of the sum. Observe
that this holds as well for $r=1$ and $r=t$ since
$\lambda_0^{2u-1} = \lambda_t^{2u} = \rho \notin \Lambda$. Due to
the unimodularity of
$g$ we have then
$\langle \pi(\lambda_{r-1}^u) g, \pi(\lambda_r^u) g \rangle = 0$ and the
corresponding $F_{\lambda_1^1,\hdots,\lambda_t^{2L}}$ does not contribute to the sum.
Hence, in the following we may assume as in Section~\ref{Sec:Cond} that $k_{r-1}^u \neq k_{r}^u$.

As in (\ref{equation:V2sumsexpectedvaluesummands}) we may write
\begin{align}
F_{\lambda_1^1\ldots\lambda_t^{2L}}
   & \, =\, \sum_{\substack{j_1^1,\hdots,j_1^{2L}=1\\j_2^1,\hdots,j_2^{2L} = 1\\ \vdots \\j_t^1,\hdots,j_{t}^{2L}=1}}^n
T_{\ell_0;\ell_1^1,\hdots,\ell_{r}^{2L}}(j_1^1,\hdots,j_t^{2L})
J_{k_0;k_1^1,\hdots,k_r^{2L}}(j_1^1,\hdots,j_t^{2L})\label{F_def}
\end{align}
with
\begin{equation}
T_{\ell_0;\ell_1^1,\hdots,\ell_r^{2L}}(j_1^1,\hdots,j_t^{2L}) \,=\, \prod_{u=1}^{2L}
\prod_{r=1}^t e^{2\pi i j_{r}^u(\ell_{r-1}^u - \ell_r^u)}\label{T_def}
\end{equation}
and
\[
J_{k_0;k_1^1,\hdots,k_r^{2L}}(j_1^1,\hdots,j_t^{2L}) \,=\, \E\left[ \prod_{u=1}^{2L} \prod_{r=1}^t
g(j_{r}^{u} - k_{r-1}^{u}) \overline{g(j_{r}^{u} - k_r^{u})} \right].
\]
As discussed in Section~\ref{Sec:Cond}, the independence of the $g(j)$ implies that the expectation above factorizes
into a product of expectations. However, we have to be careful again since
some of the indices $j_{r}^{u} - k_{r-1}^{u}$ and $j_{r'}^{u'} - k_{r'}^{u'}$ might
equal the same number $j$. In this case one of the factors in the product
equals $\E[|g(j)|^2] = 1/n$ (or a higher power if more than two
indices are equal).  As in the proof of Lemma \ref{lem:frob} we have to count
such cases. Again, they necessitate that $j_1^1,\hdots,j_t^{2L}$ must decompose into $s$ sets
\[
\{j_{\alpha_{11}},j_{\alpha_{12}},\ldots,j_{\alpha_{1r_1}} \},\ \{j_{\alpha_{21}},j_{\alpha_{22}},\ldots,j_{\alpha_{2r_2}} \},\ldots,\{j_{\alpha_{s1}},j_{\alpha_{s2}},\ldots,j_{\alpha_{sr_s}} \},
\]
with $r_1 + r_2 + \hdots + r_s = 2tL$, and for each $q=1,\hdots,s$ we have
\begin{eqnarray}
j_{\alpha_{q1}} - k_{\alpha_{q1}-1} &=& j_{\alpha_{q2}} - k_{\alpha_{q2}},\notag\\
j_{\alpha_{q2}} - k_{\alpha_{q2}-1} & = & j_{\alpha_{q3}} - k_{\alpha_{q3}},\notag\\
& \vdots & \notag\\
j_{\alpha_{qr_q}} - k_{\alpha_{qr_q}-1} & = & j_{\alpha_{q1}} - k_{\alpha_{q1}}.\label{j_set}
\end{eqnarray}
Here, it is understood that $k_{\alpha - 1} = k_{r-1}^u$ if $\alpha = (r,u)$.
As done earlier in Lemma \ref{lem:frob} we represent such a case by the $s$ cycles
\[
        {\alpha_{11}} \rightarrow  {\alpha_{12}}\rightarrow \ldots  \rightarrow {\alpha_{1r_1}}\rightarrow  {\alpha_{11} },
            \quad \ldots \quad,\, {\alpha_{s1}}   \rightarrow {\alpha_{s2}}\rightarrow \ldots  \rightarrow {\alpha_{sr_s}}\rightarrow  {\alpha_{s1}} .
\]
Now if
\begin{equation}
\sum_{p=1}^{r_q} k_{\alpha_{qp}-1} = \sum_{p=1}^{r_q} k_{\alpha_{qp}}\label{k_eq}
\end{equation}
for all $q=1,\hdots,s$ then any vector of indices
$(j_{\alpha_{11}},\hdots,j_{\alpha_{s1}}) \in \{1,\hdots,n\}^s$ gives
\[
J_{k_0;k_1^1,\hdots,k_t^{2L}}(j_1^1,\hdots,j_t^{2L}) = n^{-s}
\]
by setting the other indices $j_{r}^u$ according to (\ref{j_set}).
Plugging this into (\ref{F_def}) and (\ref{T_def}) we realize that
these contributions are canceled out unless
\begin{equation}
\sum_{p=1}^{r_q} \ell_{\alpha_{qp}-1} = \sum_{p=1}^{r_q} \ell_{\alpha_{qp}}\label{l_eq}.
\end{equation}
So we obtain $n^s$ non-zero contributions of absolute
value
$n^{-2tL}$ to $F_{\lambda_1^1,\hdots,\lambda_{t-1}^{2L}}$ if and only if
\[
\sum_{p=1}^{r_q} \lambda_{\alpha_{qp}-1} = \sum_{p=1}^{r_q} \lambda_{\alpha_{qp}} \quad \mbox{ for all } q=1,\hdots,s.
\]
Arguing similarly as in the proof of Lemma \ref{lem:frob}
we conclude that these $s$ equations are linearly independent.
However note that in contrast to the situation there we now have
$\lambda_0^{2u-1} = \lambda_{t}^{2u} = \rho \notin \Lambda$.
With similar arguments as in the end of the proof of Lemma
\ref{lem:frob} we finally
obtain
\[
\E |((K R_\Lambda)^t \sigma)_{\rho}|^{2L}
\leq \left(\frac{|\Lambda|}{n}\right)^{2tL} \sum_{s=1}^{tL} d_2(2tL,s)
\left(\frac{n}{|\Lambda|} \right).
\]
\end{proof}

\subsection{Proof of Theorem \ref{thm:recover}}

Applying Lemma \ref{lem:aux}, using the estimates
of Lemmas \ref{lem:frob} and \ref{lem:expaux}, the definition
of the function $G_{2m}$ in (\ref{def:Gm}) and $|\Lambda^c| \leq n^2$,
we conclude that
the  probability of recovery failure is upper bounded by
\begin{align}
& \kappa^{-2} \E\left[\|H^m\|_F^2\right] + \beta^{-2m}\sum_{\rho \in \Lambda^c} \sum_{t=1}^m \E\left[|((KR_\Lambda)^t \sgn(x))_\rho|^{2L_t} \right]\notag\\
&\qquad \leq
\kappa^{-2} \sparsity G_{2m}(n/\sparsity) + n^2\beta^{-2m} \sum_{t=1}^m G_{2tL_t}(n/\sparsity)\label{prob_bound}
\end{align}
provided that the conditions given in (\ref{cond_kappa00}) hold.

For specific $\sparsity,n$ one may already use this estimate to compute
an explicit probability bound by numerically minimizing over $m$ and the
remaining parameters.
Following the analysis in \cite{ra05-7}
we provide an estimate, which is easier to interpret.

We choose $L_t$ as $m/t$ rounded to the nearest integer.  It is then straightforward to deduce that
\[
t L_t \in \{\left\lceil 2m/3 \right\rceil,\left\lceil 2m/3 \right\rceil{+}1,\hdots,\left\lfloor 4m/3\right\rfloor\},
\quad t \in \{1,\hdots,m\}.
\]
Let $z = n/\sparsity$. Using (\ref{G:estim}) we obtain
\[
\sum_{t=1}^m G_{2tL_t}(z) \,\leq\,
m \max_{m' \in \{\left\lceil 2m/3 \right\rceil,\hdots, \left\lfloor 4m/3 \right\rfloor\}} G_{2m'}(z)
\,\leq\, m \max_{m' \in \{\left\lceil 2m/3 \right\rceil,\hdots, \left\lfloor 4m/3 \right\rfloor\}} \frac{\alpha^{m'}}{4(1-\alpha)}
\,\leq\, m \frac{\alpha^{2m/3}}{4(1-\alpha)}
\]
for any $\alpha <1$ with  $4m'/z \leq \alpha$ for all
$m' \in \{\left\lceil 2m/3 \right\rceil,\hdots,\left\lfloor 4m/3\right\rfloor\}$, which
is the case for
\begin{equation}\label{def:mz2}
m = m_z \,=\,  \left\lfloor \frac{3 \alpha z}{16} \right\rfloor.
\end{equation}
This yields
\begin{equation}\label{eq:E1}
n^2 \beta^{-2m} \sum_{t=1}^{m} G_{2tL_t}(z) \leq n^2 m_z \frac{(\beta^{-3}\alpha)^{2m_z/3}}{4(1-\alpha)}.
\end{equation}
Now choose
\begin{equation}\label{def:alpha}
\alpha \,:=\, \beta^3 e^{-3/2}.
\end{equation}
Then the right hand side of (\ref{eq:E1}) becomes $n^2 m_z \frac{e^{-m_z}}{4(1-\alpha)}$ which is less than or equal to $\varepsilon / 2$ if
and only if
\[
m_z - \log\left(\frac{m_z}{2(1-\alpha)}\right) \geq \log(n^2/\varepsilon).
\]
A numerical test shows that $\beta = 0.47$ is a valid choice and the corresponding $a = \sum_{t=1}^m \beta^{m/L_t}$ will always be less than $0.957$. 
Assume for the moment that $m_z \geq M \in \N$, $M \geq 6$. Since 
$t \mapsto t^{-1}\log(\frac{t}{2(1-\alpha)})$ is monotonically decreasing 
for $t \geq 6$ and $\alpha$ as in (\ref{def:alpha}), $\beta = 0.47$, we
obtain  
\[
m_z-\log\left(\frac{m_z}{2(1-\alpha)}\right) =
m_z\left(1-m_z^{-1}\log\left(\frac{m_z}{2(1-\alpha)}\right)\right)
\geq m_z\left(1- \frac{\log(M(1-\alpha)^{-1}/2)}{M}\right).
\]
The elementary inequality $\lfloor y \rfloor \geq \frac{M}{M+1} y$ for 
$y \geq M$ yields 
\[
m_z = \left\lfloor \frac{3 \alpha z}{16} \right\rfloor \geq 
\frac{3 \alpha z M}{16(M+1)} = \frac{3M}{16(M+1)} \beta^3e^{-3/2} z.
\] 
Altogether, the left hand side of (\ref{eq:E1}) is less than
$\varepsilon / 2$ provided
\begin{equation}\label{z_ineq}
\frac{n}{\sparsity} = z \geq Q(\beta,M)^{-1} \log(n^2/\varepsilon)
\end{equation}
and $m_z \geq M$, where
\[
Q(\beta,M) \,:=\, \frac{3M}{16(M+1)} \beta^3 e^{-3/2}\left(1- \frac{\log(M(1-\beta^3e^{-3/2})^{-1}/2)}{M}\right).
\]
Taking $M=20$ yields
\[
C_1 := Q(0.47,20)^{-1} \approx 273.5.
\]
Without loss of generality we may assume $S\geq 1$ (otherwise $x=0$ and
there is nothing to prove). Then (\ref{z_ineq}) requires at least
$n/\log(n^2) \geq C_1 S \geq C_1$, and a numerical test reveals that
necessarily $n\geq 10000$.
The minimal choice $z = C_1 \log(10000^2)$ yields then
$m_z = \left\lfloor 3 \alpha C_1 \log(10000^2) / 16 \right\rfloor
\,=\, 21 \geq 20 = M$, that is, our initial assumption $m_z \geq M$
is satisfied if
\begin{equation}\label{cond_n1}
n \geq C_1 \sparsity \log(n^2/\varepsilon)
\end{equation}
and, hence, this
ensures $n^2 \beta^{-2m} \sum_{t=1}^m G_{2tL_t}(n/\sparsity) \leq \varepsilon /2$
as well.

Now consider the other term $\sparsity\kappa^{-2} G_{2m}(n/\sparsity)$ in the probability
bound (\ref{prob_bound}). We choose $\kappa$ such that there is
equality in the second inequality of (\ref{cond_kappa00}), that is,
\[
\kappa \,=\, \frac{(1-a)/(1+a)\sparsity^{-3/2}}{1+(1-a)/(1+a)\sparsity^{-3/2}}
\,\geq\, \frac{1-a}{2(1+a)} \sparsity^{-3/2}.
\]
Together with (\ref{G:estim}) and the choice (\ref{def:mz2}) (with $z=n/\sparsity$)
we obtain
\[
\sparsity\kappa^{-2} G_{2m_z}(z)
\,\leq\,  \left(\frac{2(1+a)}{(1-a)}\right)^{2} \sparsity^4 \frac{\alpha^{m_z}}{4(1-\alpha)}.
\]
Requiring that the latter expression is less
than $\varepsilon/2$ is equivalent to
\[
\log(\alpha^{-1}) m_z \,\geq\, \log(\sparsity^4/\varepsilon) + \log\left(2\frac{(1+a)^2}{(1-a)^2(1-\alpha)}\right).
\]
As above assume for the moment that $m_z \geq M$. Plugging in
$\alpha$ from above yields
\[
m_z \,\geq\, \frac{M}{M+1} \frac{3\alpha z}{16} \,=\,
\frac{3M}{16(M+1)} \beta^{3} e^{-3/2} z.
\]
It follows that $\sparsity\kappa^{-2} G_{2m_z}(z) \leq \varepsilon/2$ if
\[
z \,\geq\, \frac{16(M+1)\beta^{-3}e^{3/2}}{3M\log(\beta^{-3}e^{3/2})}
\left(\log(\sparsity^4/\varepsilon) +  \log\left(2\frac{(1+a)^2}{(1-a)^2}(1-\beta^3e^{-3/2})^{-1}\right)\right).
\]
As already remarked the choice $\beta=0.47$ results in $a \leq 0.957$.
Choosing $M=21$ ($m_z \geq 21$ will be ensured by (\ref{cond_n1}) anyway
as shown above) gives
\begin{equation}\label{cond_n2}
z \,\geq\, C_2 (\log(\sparsity^4/ \varepsilon) + C_3)
\end{equation}
with $C_2 \approx 64.1$ and $C_3 \approx 8.35$.

Since $\sparsity \leq n^2$, combining (\ref{cond_n1}) and (\ref{cond_n2})
finally shows the existence of a constant $C$ such that
\[
n \geq C \sparsity \log(n/\varepsilon)
\]
ensures recovery with probability at least $1-\varepsilon$.
This proves Theorem~\ref{thm:recover}.


\begin{thebibliography}{10}

\bibitem{al80}
W.~O. {A}lltop.
\newblock {C}omplex sequences with low periodic correlations.
\newblock {\em {I}{E}{E}{E} {T}rans. {I}nf. {T}heory}, 26(3):350--354, 1980.

\bibitem{badadewa06}
R.~{B}araniuk, M.~{D}avenport, R.~{D}e{V}ore, and M.~{W}akin.
\newblock {A} simple proof of the restricted isometry property for random
  matrices.
\newblock {\em {C}onstr. {A}pprox.}, to appear.

\bibitem{be63}
P.~{B}ello.
\newblock {C}haracterization of randomly time-variant linear channels.
\newblock {\em {I}{E}{E}{E} {T}rans. {C}ommun.}, 11:360--393, 1963.

\bibitem{bova04}
S.~{B}oyd and L.~{V}andenberghe.
\newblock {\em {C}onvex Optimization.}
\newblock {C}ambridge {U}niv. {P}ress, 2004.

\bibitem{carota06}
E.~{C}and{\`e}s, J.~{R}omberg, and T.~{T}ao.
\newblock {R}obust uncertainty principles: exact signal reconstruction from
  highly incomplete frequency information.
\newblock {\em {I}{E}{E}{E} {T}rans. {I}nf. {T}heory}, 52(2):489--509, 2006.

\bibitem{carota06-1}
E.~{C}and{\`e}s, J.~{R}omberg, and T.~{T}ao.
\newblock {S}table signal recovery from incomplete and inaccurate measurements.
\newblock {\em {C}omm. {P}ure {A}ppl. {M}ath.}, 59(8):1207--1223, 2006.

\bibitem{cata06}
E.~{C}and{\`e}s and T.~{T}ao.
\newblock {N}ear optimal signal recovery from random projections: universal
  encoding strategies?
\newblock {\em {I}{E}{E}{E} {T}rans. {I}nf. {T}heory}, 52(12):5406--5425, 2006.

\bibitem{chdosa99}
S.~{C}hen, D.~{D}onoho, and M.~{S}aunders.
\newblock {A}tomic decomposition by {B}asis {P}ursuit.
\newblock {\em {S}{I}{A}{M} {J}. {S}ci. {C}omput.}, 20(1):33--61, 1999.

\bibitem{Chr03}
O.~Christensen.
\newblock {\em An {I}ntroduction to {F}rames and {R}iesz {B}ases}.
\newblock Applied and Numerical Harmonic Analysis. Birkh\"auser Boston Inc.,
  Boston, MA, 2003.

\bibitem{Cor01}
L.~M. Correia.
\newblock {\em Wireless Flexible Personalized Communications}.
\newblock John Wiley \& Sons, Inc., New York, NY, USA, 2001.

\bibitem{avdama97}
G.~{D}avis, S.~{M}allat, and M.~{A}vellaneda.
\newblock {A}daptive greedy approximations.
\newblock {\em {C}onstr. {A}pprox.}, 13(1):57--98, 1997.

\bibitem{do04}
D.~{D}onoho.
\newblock {C}ompressed sensing.
\newblock {\em {I}{E}{E}{E} {T}rans. {I}nf. {T}heory}, 52(4):1289--1306, 2006.

\bibitem{dots06-1}
D.~{D}onoho and Y.~{T}saig.
\newblock {F}ast solution of l1-norm minimization problems when the solution
  may be sparse.
\newblock {\em Preprint}, 2006.

\bibitem{doelte06}
D.~L. {D}onoho, M.~{E}lad, and V.~N. {T}emlyakov.
\newblock {S}table recovery of sparse overcomplete representations in the
  presence of noise.
\newblock {\em {I}{E}{E}{E} {T}rans. {I}nf. {T}heory}, 52(1):6--18, 2006.

\bibitem{fu04}
J.-J. {F}uchs.
\newblock {O}n sparse representations in arbitrary redundant bases.
\newblock {\em {I}{E}{E}{E} {T}rans. {I}nform. {T}heory}, 50(6):1341--1344,
  2004.

\bibitem{gitr05}
A.~{G}ilbert and J.~{T}ropp.
\newblock {S}ignal recovery from random measurements via orthogonal matching
  pursuit.
\newblock {\em IEEE Trans. Inform. Theory}, to appear.

\bibitem{GP07}
N.~{G}rip and G.~{P}fander.
\newblock A discrete model for the efficient analysis of time-varying
  narrowband communication channels.
\newblock {\em Multidim. Syst. Sign. P.}, to appear.

\bibitem{Gro01}
K.~Gr{\"o}chenig.
\newblock {\em Foundations of Time-Frequency Analysis}.
\newblock Applied and Numerical Harmonic Analysis. Birkh{\"a}user, Boston, MA,
  2001.

\bibitem{grpora07}
K.~{G}r{\"o}chenig, B.~{P}{\"o}tscher, and H.~{R}auhut.
\newblock {L}earning trigonometric polynomials from random samples and
  exponential inequalities for eigenvalues of random matrices.
\newblock {\em {P}reprint}, 2007.

\bibitem{HS07}
M.~Herman and T.~Strohmer.
\newblock High resolution radar via compressed sensing.
\newblock {\em {P}reprint}, 2007.

\bibitem{krpfra07}
F.~{K}rahmer, G.~{P}fander, and P.~{R}ashkov.
\newblock {U}ncertainty principles for time--frequency representations on
  finite abelian groups.
\newblock {\em Appl. Comput. Harmon. Anal.}
\newblock to appear.

\bibitem{kura06}
S.~{K}unis and H.~{R}auhut.
\newblock {R}andom sampling of sparse trigonometric polynomials {I}{I} -
  orthogonal matching pursuit versus basis pursuit.
\newblock {\em Found. Comput. Math.}, to appear.

\bibitem{LPW05}
J.~Lawrence, G.~Pfander, and D.~Walnut.
\newblock Linear independence of {G}abor systems in finite dimensional vector
  spaces.
\newblock {\em J. Fourier Anal. Appl.}, 11(6):715--726, 2005.

\bibitem{mi87}
D.~Middleton.
\newblock Channel modeling and threshold signal processing in underwater
  acoustics: An analytical overview.
\newblock {\em IEEE J. Oceanic Eng.}, 12(1):4--28, 1987.

\bibitem{Pae01}
M.~P\"atzold.
\newblock {\em Mobile Fading Channels: Modelling, Analysis and Simulation}.
\newblock John Wiley \& Sons, Inc., 2001.

\bibitem{pfrata07}
G.~E. {P}fander, H.~{R}auhut, and J.~{T}anner.
\newblock {I}dentification of matrices having a sparse representation.
\newblock {\em {P}reprint}, 2007.

\bibitem{ra05-7}
H.~{R}auhut.
\newblock {R}andom sampling of sparse trigonometric polynomials.
\newblock {\em {A}ppl. {C}omput. {H}arm. {A}nal.}, 22(1):16--42, 2007.

\bibitem{rascva06}
H.~{R}auhut, K.~{S}chnass, and P.~{V}andergheynst.
\newblock {C}ompressed sensing and redundant dictionaries.
\newblock {\em {P}reprint}, 2006.

\bibitem{ri58}
J.~{R}iordan.
\newblock {\em {A}n {I}ntroduction to {C}ombinatorial {A}nalysis.}
\newblock {W}iley, {N}ew {Y}ork, {L}ondon, 1958.

\bibitem{ruve06}
M.~{R}udelson and R.~{V}ershynin.
\newblock {O}n sparse reconstruction from {F}ourier and {G}aussian
  measurements.
\newblock {\em {C}omm. {P}ure {A}ppl. {M}ath.}, to appear.

\bibitem{st99}
M.~Stojanovic.
\newblock {\em Underwater Acoustic Communications}, volume~22, pages 688--698.
\newblock John Wiley {\&} Sons, 1999.

\bibitem{Str98}
T.~Strohmer.
\newblock Numerical algorithms for discrete {G}abor expansions.
\newblock In {\em Gabor analysis and algorithms}, Appl. Numer. Harmon. Anal.,
  pages 267--294. Birkh\"auser Boston, Boston, MA, 1998.

\bibitem{hest03}
T.~{S}trohmer and R.~W. {H}eath.
\newblock {G}rassmannian frames with applications to coding and communication.
\newblock {\em {A}ppl. {C}omput. {H}armon. {A}nal.}, 14(3):257--275, 2003.

\bibitem{tr04}
J.~{T}ropp.
\newblock {G}reed is good: {A}lgorithmic results for sparse approximation.
\newblock {\em {I}{E}{E}{E} {T}rans. {I}nf. {T}heory}, 50(10):2231--2242, 2004.

\bibitem{tr05-1}
J.~A. {T}ropp.
\newblock {R}ecovery of short, complex linear combinations via $l_1$
  minimization.
\newblock {\em {I}{E}{E}{E} {T}rans. {I}nform. {T}heory}, 51(4):1568--1570,
  2005.

\bibitem{tr06-2}
J.~A. {T}ropp.
\newblock {O}n the conditioning of random subdictionaries.
\newblock {\em {A}ppl. {C}omput. {H}armon. {A}nal.}, to appear.

\end{thebibliography}
\end{document}